\documentclass[a4paper, 10pt]{amsart}
\usepackage{amsmath}
\usepackage{amssymb}
\usepackage{color}
\usepackage{hyperref}
\usepackage[inline]{enumitem}
\usepackage{mathrsfs}
\usepackage{stmaryrd}
\usepackage{comment}
\usepackage[normalem]{ulem}

\theoremstyle{plain}
\newtheorem{theorem}{\bf Theorem}[section]
\newtheorem{proposition}[theorem]{\bf Proposition}
\newtheorem{lemma}[theorem]{\bf Lemma}
\newtheorem{basic-lemma}[theorem]{\bf Basic Lemma}
\newtheorem{claim}{\bf Claim}

\theoremstyle{definition}
\newtheorem{example}[theorem]{\bf Example}

\newtheorem{remark}[theorem]{\bf Remark}
\newtheorem{problem}[theorem]{\bf Problem}

\newcommand{\N}{\mathbb N}
\newcommand{\Z}{\mathbb Z}
\newcommand{\R}{\mathbb R}
\newcommand{\Q}{\mathbb Q}

\renewcommand{\P}{\mathbb P}
\newcommand{\DP}{\negthinspace : \negthinspace}

\renewcommand{\t}{\, | \,}
\newcommand{\FF}{\text{\rm FF}}
\newcommand{\BF}{\text{\rm BF}}

\providecommand\llb{\llbracket}
\providecommand\rrb{\rrbracket}

\DeclareMathOperator{\lcm}{lcm}

\newcommand{\red}{{\text{\rm red}}}
\newcommand{\mon}{{\text{\rm mon}}}

\numberwithin{equation}{section}

\hypersetup{
	pdftitle={Strongly primary monoids},
	pdfauthor={Geroldinger, Gotti, and Tringali},
	pdfmenubar=false,
	pdffitwindow=true,
	pdfstartview=FitH,
	colorlinks=true,
	linkcolor=blue,
	citecolor=magenta,
	urlcolor=cyan
}

\begin{document}

\title{On strongly primary monoids, with a focus on Puiseux monoids}

\author{Alfred Geroldinger}
\address{Institute of Mathematics and Scientific Computing\\ University of Graz,  Heinrichstr. 36\\ 8010 Graz, Austria }
\email{alfred.geroldinger@uni-graz.at}
\urladdr{https://imsc.uni-graz.at/geroldinger}

\author{Felix Gotti}
\address{Department of Mathematics, UC Berkeley, Berkeley, CA 94720, USA}
\curraddr{Department of Mathematics, Massachusetts Institute of Technology, Cambridge, MA 02139, USA}
\email{felixgotti@berkeley.edu, fgotti@mit.edu}
\urladdr{https://felixgotti.com}

\author{Salvatore Tringali}
\address{Institute of Mathematics and Scientific Computing, University of Graz, Heinrichstr. 36, 8010 Graz, Austria}
\curraddr{School of Mathematical Sciences,
	Hebei Normal University, Shijiazhuang, Hebei province, 050024 China}
\email{salvatore.tringali@uni-graz.at, salvo.tringali@gmail.com}
\urladdr{http://imsc.uni-graz.at/tringali}

\thanks{This work was supported by the Austrian Science Fund FWF, Project Number P33499-N. The second author was partially supported by the NSF award DMS-1903069.}



\begin{abstract}
Primary and strongly primary monoids and domains play a central role in the ideal and factorization theory of commutative monoids and domains. It is well-known that primary monoids satisfying the ascending chain condition on divisorial ideals (e.g., numerical monoids) are strongly primary; and the multiplicative monoid of non-zero elements of a one-dimensional local domain  is primary and it is strongly primary if the domain is Noetherian. In  the present paper, we focus on the study of additive submonoids of the non-negative rationals, called Puiseux monoids. It is easy to see that Puiseux monoids are primary monoids, and we provide conditions ensuring that they are strongly primary. Then we study local and global tameness of strongly primary Puiseux monoids; most notably, we  establish an algebraic characterization of when a Puiseux monoid is globally tame. Moreover, we obtain a result on the structure of
 sets of lengths of all locally tame strongly primary monoids.
\end{abstract}

\maketitle
\bigskip

\section{Introduction}
\label{sec:intro}

\medskip
A (commutative and cancellative) monoid $H$ is primary if it is not a group and if for each two non-invertible elements $a, b \in H$ there is $n \in \N$ such that $b^n \in aH$. Furthermore, $H$ is strongly primary if for each non-invertible element $a \in H$ there is $n \in \N$ such that $\mathfrak m^n \subset aH$, where $\mathfrak m $ is the non-empty set of non-invertible elements of $H$. A domain $R$ is (strongly) primary if its multiplicative monoid of non-zero elements is (strongly) primary. Thus, a domain is primary if and only if it is one-dimensional and local. Every primary Mori monoid is strongly primary, which implies that numerical monoids and one-dimensional local Mori domains are strongly primary. If $R$ is a weakly-Krull Mori domain, then its localizations at height-one prime ideals are strongly primary, whence its monoid of ($v$-invertible divisorial) ideals is a coproduct of strongly primary monoids. So we see that (strongly) primary monoids play a central role in the ideal and factorization theory of commutative monoids and domains.

The goal of the present paper is twofold. In Section \ref{sec:Puiseux monoids} we concentrate on Puiseux monoids, that is, additive submonoids of the non-negative rationals. These are primary monoids  generalizing numerical monoids, and have attracted quite a lot of attention in recent literature, also in relation to the study of semigroup algebras (e.g., \cite{Co-Go19a, Go19c}). In Subsections \ref{subsec:conductor} and \ref{subsec:strongly-primary}, we inquire into the algebraic properties of Puiseux monoids, with a focus on conductors and conditions enforcing strong primariness. Then we investigate the arithmetic properties of strongly primary Puiseux monoids, with emphasis on local and global tameness. To recall these concepts, let $H$ be an atomic monoid. The local tame degree $\mathsf t (H,u)$ (of an atom $u$) is the smallest $N \in \N_0 \cup \{\infty\}$ such that, in any given factorization of a multiple $a \in uH$ of $u$,  at most~$N$ atoms have to be replaced by at most~$N$ new atoms to obtain a  factorization of $a$ that contains $u$. The global tame degree $\mathsf t (H)$ is the supremum of the local tame degrees over all atoms, and we say that $H$ is globally tame if $\mathsf t (H) < \infty$. We provide a characterization of global tameness (Theorem \ref{thm:a SPPM has finite elasticity iff it is globally tame}), and we carefully work out the similarities and differences between the arithmetic of strongly primary Puiseux monoids on one side and the arithmetic of strongly primary domains and finitely primary monoids on the other side. Theorem \ref{thm:Puiseux-monoids-and-their-quotient-groups} shows that the cardinality of the class consisting of strongly primary Puiseux monoids that are either globally tame or locally but not globally tame is at least as large as the cardinality of the class consisting of additive subgroups of the rationals.

In Section \ref{sec:sets of lengths of LTSPM} we study the arithmetic of locally tame strongly primary monoids, with a focus on their sets of lengths. There are Puiseux monoids having every finite set $L \subset \N_{\ge 2}$ as a set of lengths. In contrast to this, sets of lengths of locally tame strongly primary monoids are well structured. Our main result on sets of lengths is Theorem \ref{thm:structure-of-sets-of-lengths-and-unions}. Beyond that we provide the first example of a primary BF-monoid having irrational elasticity (Example~\ref{exa:strongly-primary-with-irrational-elasticity}).
\medskip

\section{Background on primary monoids}
\label{sec:background}

\smallskip
\subsection{General Notation} We denote by $\mathbb P$, $\N$,  $\N_0$,  $\Z$, $\Q $, and $\R$ the set of prime numbers, positive integers, non-negative integers, integers, rational numbers, and real numbers, respectively. For $a, b \in \R \cup \{\pm \infty \}$, we let $\llb a, b \rrb$ denote the discrete interval $\{ x \in \Z : a \le x \le b \}$, and we let $[a, b]$, $]a, b]$, etc. denote the usual real intervals. Let~$G$ be an abelian group, and take $A, B \subset G$. Then $A + B = \{a+b : a \in A, b \in B \}$ denotes the sumset and $A - B = \{a-b : a \in A, b \in B \}$ denotes the difference set of $A$ and $B$. In addition, for $k \in \N$, we let $kA = A+ \dots + A$ and $k \cdot A = \{ka : a \in A\}$ denote the $k$-fold sumset of $A$ and the dilation of $A$ by $k$, respectively. For a non-empty subset $L \subset \N$, we call  $\rho (L) = \sup L/ \min L$ the {\it elasticity} of $L$ and we set $\rho ( \{0\}) = 1$.

\smallskip
\subsection{Monoids} Throughout this paper, the term `monoid' refers to a commutative cancellative semigroup with identity element while the term `domain' refers to a commutative ring with identity and without non-zero zero divisors. In the present section and in Section~\ref{sec:sets of lengths of LTSPM}, we use multiplicative notation because some of the main examples we have in mind stem from ring theory (note that the multiplicative subset $R^{\bullet} = R\setminus \{0\}$ of a domain $R$ is a monoid). Only in Section~\ref{sec:Puiseux monoids}, where we study additive submonoids of the non-negative rationals, we use additive notation. We briefly recall some ideal-theoretic and arithmetic concepts in the context of monoids.

Let $H$ be a monoid. Then $H^{\times}$ denotes the group of invertible elements, $H_{\red} = H/H^{\times}$ the associated reduced monoid, and $\mathsf q (H)$ the quotient group of $H$.
We define
\begin{itemize}
	\item $H' = \{ x \in \mathsf q (H) : \text{there exists some $N \in \N$ such that $x^n \in H$ for all $n \ge N$} \}$,
	\item $\widetilde H = \{ x \in \mathsf q (H) : \text{there exists some $N \in \N$ such that $x^N \in H$} \}$,  and
	\item $\widehat H = \{ x \in \mathsf q (H) : \text{there exists $c \in H$ such that $cx^n \in H$ for all $n \in \N$} \}$;
\end{itemize}
and we call $H'$ the {\it seminormal closure}, $\widetilde H$ the {\it root closure}, and $\widehat H$ the {\it complete integral closure} of $H$, respectively. Then we have
\begin{equation} \label{eq2.1}
	H \subset H' \subset \widetilde H \subset \widehat H \subset \mathsf q (H) \,,
\end{equation}
where all inclusions can be strict. We say that $H$ is {\it seminormal} if $H=H'$, {\it root-closed} if $H = \widetilde H$, and {\it completely integrally closed} if $H = \widehat H$. Note that $H'$ is seminormal, $\widetilde H$ is root-closed, but $\widehat H$ may not be completely integrally closed (\cite[Theorem~3]{Ge-HK-Le95}). However, if the conductor $(H \DP \widehat H) = \{ x \in \mathsf q(H) : x \widehat H \subset H \}$ is non-empty, then $\widehat H$ is completely integrally closed.
The monoid $H$ is said to be
\begin{itemize}
	\item {\it $v$-Noetherian} (or a {\it Mori monoid}) if it satisfies the ascending chain condition on divisorial ideals,
	\item a {\it  Krull monoid} if it is a completely integrally closed Mori monoid,
	\item a {\it valuation monoid} if for all $a, b \in H$ either $a \mid b$ or $b \mid a$,
	\item {\it primary} if  $H \ne H^{\times}$ and for all $a, b \in H \setminus H^{\times}$ there exists $n \in \N$ such that $b^n \in aH$, and
	\item {\it strongly primary} if $H \ne H^{\times}$ and for each $a \in H \setminus H^{\times}$ there exists $n \in \N$ such that $(H \setminus H^{\times})^n \subset aH$. We let $\mathcal M (a)$ denote the smallest $n \in \N$ having this property, and we set
	\[
		\mathcal M(H) = \sup \{ \mathcal M(u) : u \in \mathcal A (H)\}.
	\]
\end{itemize}
Strongly primary monoids are the objects of study in the present paper, and we discuss them in Subsection~\ref{background-strongly-primary}.
If $H$ is primary, then it follows from \cite[Proposition~1]{Ge96} and \cite[Lemma~2.5]{Ge-Ro19b} that
\begin{equation} \label{primary-closure}
	H' = \widetilde H \quad \text{and} \quad H^{\times} =  \widetilde H^{\times} \cap H = \widehat H^{\times} \cap H \,.
\end{equation}
Further, we use that a valuation monoid $H$ is completely integrally closed if and only if either $H=H^{\times}$ or $H$ is primary (\cite[Chapter~16.3]{HK98}). For a set $P$, we denote by $\mathcal F (P)$ the free commutative monoid with basis $P$. Every element $a \in \mathcal F (P)$ can be written uniquely in the form
\[
	a = \prod_{p \in P} p^{\mathsf v_p (a)} \,,
\]
where $\mathsf v_p (a) \in \N_0$ is the $p$-adic valuation of $a$, and we call $|a| = \sum_{p \in P} \mathsf v_p (a)$ the {\it length} of $a$.

\smallskip
\subsection{Arithmetic of monoids} We proceed to gather the arithmetic concepts needed in the sequel. For details and proofs we refer to \cite[Chapter~1]{Ge-HK06a}. We denote by $\mathsf Z(H) = \mathcal F ( \mathcal A (H_{\red}))$ the factorization monoid of $H$ and let $\pi \colon \mathsf Z(H) \to H_{\red}$ be the factorization homomorphism. For an element $a \in H$,
\begin{itemize}
	\item $\mathsf Z(a) = \pi^{-1} (aH^{\times})$ is the {\it set of factorizations} of $a$,
	\item $\mathsf L(a) = \{ |z| : z \in \mathsf Z (a) \} \subset \N_0$ is the {\it set of lengths} of $a$, and
    \item $\mathscr L(H) = \{\mathsf L(a) : a \in H \}$ is  the {\it system of sets of lengths} of $H$.
\end{itemize}
Note that $\mathsf L(a) = \{0\}$ if and only if $a \in H^{\times}$, and that $1 \in \mathsf L (a)$ if and only if $a \in \mathcal A (H)$. We say that $H$ is
\begin{itemize}
	\item {\it atomic} if $\mathsf Z(a) \ne \emptyset$ for every $a \in H$,
	\item a {\it \BF-monoid} if $\mathsf L(a)$ is finite and non-empty for every $a \in H$, and
	\item an {\it \FF-monoid} if $\mathsf Z(a)$ is finite and non-empty for every $a \in H$.
\end{itemize}
Our focus in this section and in Section~\ref{sec:Puiseux monoids} is on (local and global) tameness. In Section~\ref{sec:sets of lengths of LTSPM} we show that locally tame strongly primary monoids share a variety of further arithmetic finiteness properties. For the remainder of this subsection, suppose that $H$ is an atomic monoid. To introduce the concept of tameness, we first need a distance function. Consider two factorizations $z, z' \in \mathsf Z (a)$, say
\[
	z = u_1 \cdot \ldots \cdot u_{\ell}v_1 \cdot \ldots \cdot v_m \quad \text{and} \quad z' = u_1 \cdot \ldots \cdot u_{\ell}w_1 \cdot \ldots \cdot w_n,
\]
where $\ell, m, n \in \N_0$ and all $u_i, v_j, w_k$ are in $\mathcal A (H_{\red})$ with $\{v_1, \dots, v_m\} \cap \{w_1, \dots, w_n\} = \emptyset$. Then we call $\mathsf d (z,z') = \max \{m,n\} \in \N_0$ the \emph{distance} between $z$ and $z'$. For an atom $u \in \mathcal A (H_{\red})$, the {\it local tame degree} $\mathsf t (H,u) \in \N_0$ is the smallest $N \in \N_0 \cup \{\infty\}$ with the following property:
\begin{itemize}
	\item[] If $\mathsf Z (a) \cap u \mathsf Z (H) \ne \emptyset$ and $z \in \mathsf Z (a)$, then there exists  $z' \in \mathsf Z (a) \cap u \mathsf  Z (H)$ such that $\mathsf d (z,z') \le N$.
\end{itemize}
We say that $H$ is
\begin{itemize}
	\item {\it locally tame} if $\mathsf t (H,u) < \infty$ for all $u \in \mathcal A (H_{\red})$, and
	\item {\it \textup{(}globally\textup{)} tame} if the {\it tame degree} $\mathsf t (H) = \sup \{\mathsf  t(H,u) : u \in \mathcal A (H_{\red}) \}$ is finite.
\end{itemize}
If $u$ is a prime, then every factorization of a multiple $a \in uH$ of $u$ contains $u$, whence $\mathsf t(H, u) = 0$.
The atomic monoid $H$ is factorial if and only if all its atoms are prime, whence $H$ is factorial if and only if $\mathsf t (H)=0$. If $u$ is not a prime, then $\mathsf t (H,u) \in \N_0$ is the smallest $N \in \N_0 \cup \{\infty\}$ with the following property:
\begin{enumerate}
	\item[]
	If \ $m \in \N$ and $v_1, \dots , v_m \in \mathcal A(H)$ are
	such that \ $u \t v_1 \cdot \ldots \cdot v_m$, but $u$ divides no
	proper subproduct of $v_1\cdot \ldots \cdot v_m$, then there exist
	$\ell \in \N$ and  $u_2, \ldots , u_{\ell} \in \mathcal A
	(H)$ such that $v_1 \cdot \ldots \cdot v_m = u u_2 \cdot \ldots
	\cdot u_{\ell}$ and \ $\max \{\ell ,\, m \} \le N$ (in other words, $\max \{1+\min \mathsf L (u^{-1}v_1\cdot \ldots\cdot v_m), m\} \le N$).
\end{enumerate}
The property of being tame is studied via the invariant
\[
	\Lambda (H) = \sup \{\min \mathsf L(a) : a \in H \} \in \N_0 \cup \{\infty\}
\]
and via the {\it local elasticities} $\rho_k (H)$ for $k \in \N$. To recall their definition, let us fix $k \in \N$. Then $\rho_k (H) = k$ if $H = H^{\times}$, and
\[
	\rho_k (H) = \sup \{ \sup L : L \in \mathscr L(H) \ \text{with} \ k \in L \} \in \N_{\ge k} \cup \{\infty\} \ \text{ otherwise}.
\]
As a result, it is not hard to verify that
\begin{equation} \label{eq:elasticity in terms of local elasticities}
	\rho (H) := \sup \{\rho (L) : L \in \mathscr L(H) \} = \lim_{k \to \infty} \frac{\rho_k (H)}{k}.
\end{equation}
We call $\rho(H)$ the {\it elasticity} of $H$, and we say that $H$ has {\it accepted elasticity} if there is  $L  \in \mathscr L (H)$ such that $\rho (L) = \rho (H)$.
If $H$ is not factorial, then (by \cite[Theorem~1.6.6]{Ge-HK06a})
\begin{equation} \label{elasticity-tameness}
	\max \{2, \rho (H) \} \le \mathsf t (H)\,.
\end{equation}
For an atom $u \in \mathcal A(H)$, let
$\tau (H,u) $ be the smallest $N \in \N_0 \cup \{\infty\}$ with the following property:
\begin{itemize}
    \item[] If $m \in \N$ and $v_1, \ldots, v_m \in \mathcal A (H)$ such that $u$ divides the product $v_1 \cdot \ldots \cdot v_m$ but no proper subproduct, then
        $\min \mathsf L \left( u^{-1}v_1 \cdot \ldots \cdot v_m \right) \le N$.
\end{itemize}
If $H$ is strongly primary, then
\[
	\begin{aligned}
	\min \mathsf L \left( u^{-1}v_1 \cdot \ldots \cdot v_m \right) & \le \Lambda (H) \,, \quad m < \mathcal M (u) \,, \quad \text{and} \\
	\min \mathsf L \left( u^{-1}v_1 \cdot \ldots \cdot v_m \right) & \le \max \mathsf L \left( u^{-1}v_1 \cdot \ldots \cdot v_m \right) \le \max \mathsf L (v_1 \cdot \ldots \cdot v_m) \le \rho_{\mathcal M (u)-1} (H) \,.
	\end{aligned}
\]
Thus, we have the following lemma.

\smallskip
\noindent
\begin{basic-lemma}  
Let $H$ be a strongly primary monoid.
	\begin{enumerate}
	\item  Then \begin{equation} \label{local-tameness-versus-tau}
	       \text{$H$ is locally tame if and only if $\tau (H,u) < \infty$ for all $u \in \mathcal A (H)$}\,.
	       \end{equation}
	\item If
	      \begin{equation} \label{enforcing-local-tameness}
		  \text{either} \quad     \Lambda (H) < \infty \quad \text{or} \quad \rho_k (H) < \infty \ \text{for all $k \in \N$} \,,
	     \end{equation}
	     then $H$ is locally tame.
	\end{enumerate}
\end{basic-lemma}

\smallskip
\subsection{Strongly primary monoids} \label{background-strongly-primary} Whereas primary monoids need not be atomic (e.g., $(\Q_{\ge 0}, +)$ is primary but not atomic), strongly primary monoids satisfy the stronger condition of being BF-monoids.

\begin{lemma} \label{lem:characterization-strongly-primary}
	Let $H$ be a monoid with non-empty maximal ideal $\mathfrak m = H \setminus H^{\times}$. Then the following statements are equivalent.
	\begin{enumerate}
		\item[(a)] $H$ is strongly primary.
		\item[(b)] $H$ is a \BF-monoid and for all $u \in \mathcal A (H)$ there is  $n \in \N$ such that $\mathfrak m^n \subset uH$.
		\item[(c)] $H$ is primary and there exist $a \in \mathfrak m$ and $n \in \N$ such that $\mathfrak m^n \subset aH$.
	\end{enumerate}
\end{lemma}

\begin{proof}
	(a) $\Rightarrow$ (b) If $a \in H \setminus H^{\times}$ and $a = u_1 \cdot \ldots \cdot u_k$ with $k \in \N$ and $u_1, \dots, u_k \in H \setminus H^{\times}$, then $k \le \mathcal M(a)$. Thus, $a$ can be written as a product of atoms, and $\sup \mathsf L(a) \le \mathcal M(a)$.
	
	 (b) $\Rightarrow$ (c) It is obvious.
	
	(c) $\Rightarrow$ (a) For every $b \in \mathfrak m$ there is $m \in \N$ such that $a^m  \subset bH$, whence $\mathfrak m^{mn} \subset a^mH \subset bH$.
\end{proof}

Thus, strongly primary monoids are primary BF-monoids.
However, there are seminormal primary BF-monoids and primary FF-monoids that are not strongly primary (see \cite[Example~4.7]{Ge-Ha08a} and Example~\ref{ex:when empty conductor SP/dense/BF/ACCP are not equivalent for PMs}, respectively). Example~\ref{ex:when empty conductor SP/dense/BF/ACCP are not equivalent for PMs} provides (primary) Puiseux monoids that are BF but not strongly primary.

To discuss examples of strongly primary monoids, let us first mention that all primary Mori monoids (\cite[Lemma~3.1]{Ge-Ha-Le07}) are strongly primary. Moreover, saturated submonoids of primary monoids, strongly primary monoids, primary Mori monoids are respectively primary, strongly primary, and  primary Mori, unless they are groups (\cite[Proposition~2.4.4 and Theorem~2.7.3]{Ge-HK06a}). Main examples of strongly primary monoids stem from ring theory. Let $R$ be a domain. Then its multiplicative monoid $R^{\bullet}$ of non-zero elements is primary if and only if $R$ is one-dimensional and local. If $R$ is a one-dimensional local Mori domain, then $R^{\bullet}$ is a primary Mori monoid and, therefore, strongly primary. If $R^{\bullet}$ is strongly primary, then it follows from \cite[Theorems~3.8 and~3.9]{Ge-Ro19b} that
\begin{equation} \label{strongly-primary-domains-are-locally-tame}
	\text{either} \quad \Lambda (R^{\bullet}) < \infty \quad \text{or} \quad \rho_k (R^{\bullet}) < \infty \ \text{for all $k \in \N$} \,,
\end{equation}
whence $R^{\bullet}$ is locally tame by \eqref{enforcing-local-tameness}, and we have the following characterization.

\noindent
\begin{theorem}[Characterization of global tameness for strongly primary domains] \label{global-tameness-domains} \cite[Theorems~3.8 and~3.9]{Ge-Ro19b}
	For a strongly primary domain, the following statements are equivalent.
	\begin{enumerate}
		\item[(a)] $R^{\bullet}$ is globally tame.
		\item[(b)] $\rho (R^{\bullet}) < \infty$.
		\item[(c)] $\rho_k (R^{\bullet}) < \infty$ for all $k \in \N$.
		\item[(d)] $\Lambda (R^{\bullet}) = \infty$.
	\end{enumerate}
\end{theorem}

Finitely primary monoids are the best investigated class of strongly primary monoids. We recall the definition. A monoid $H$ is said to be {\it finitely primary}  (of rank $s \in \N$ and exponent $\alpha \in \N$) if $H$ is a submonoid of a factorial monoid $F = F^{\times} \times \mathcal F (\{p_1, \dots, p_s\})$ satisfying
\[
	H \setminus H^{\times} \subset p_1 \cdot \ldots \cdot p_s F \quad \text{and} \quad (p_1 \cdot \ldots \cdot p_s)^{\alpha}F \subset H \,.
\]
If this holds, then $H^{\times} = F^{\times} \cap H$,  $\widehat H = F$, and $(H \DP \widehat H)\ne \emptyset$ (note that finitely primary monoids need not be Mori; see \cite{HK-Ha-Ka04}). If $H$ is finitely primary (with $H \subset F$ as above), then it follows from \cite[Theorem~3.1.5]{Ge-HK06a} that
\begin{equation}  \label{finitely-primary-monoids-are-locally-tame}
	\text{either} \quad \Lambda (H) < \infty \quad \text{or} \quad \rho_k (H) < \infty \ \text{for all $k \in \N$} \,,
\end{equation}
whence $H$ is locally tame by \eqref{enforcing-local-tameness}, and we have the following characterization.

\noindent
\begin{theorem}[Characterization of global tameness for finitely primary monoids] \label{global-tameness-finitely-primary}\cite[Theorems~2.9.2 and~3.1.5]{Ge-HK06a}
	For a finitely primary monoid $H$, the following statements are equivalent.
	\begin{enumerate}
		\item[(a)] $H$ is globally tame.
		\item[(b)] $\rho (H) < \infty$.
		\item[(c)] $\rho_k (H) < \infty$ for all $k \in \N$.
		\item[(d)] $\Lambda (H) = \infty$.
		\item[(e)] $s=1$.
	\end{enumerate}
\end{theorem}

It is clear that an additive submonoid $H$ of $(\N_0^s, +)$, with $H \setminus \{0\} \subset \N^s$ and $(f, \ldots , f)+ \N_0^s \subset H$ for some $f \in \N$, is finitely primary of rank $s$. An additive submonoid $H$ of $(\N_0, +)$ with finite complement $\N_0 \setminus H$ is called a {\it numerical monoid}. Numerical monoids are finitely generated and finitely primary of rank one, so they are strongly primary. Finitely primary monoids occur naturally in ring theory (\cite[Proposition~2.10.7]{Ge-HK06a}) and in module theory (\cite[Theorems~5.1 and~5.3]{Ba-Ge-Gr-Sm15}).
If $H \subset F$ is finitely primary (with all notation as above), then
\[
	\mathsf v(H) = \{ (n_1, \dots, n_s) \in \N_0^s : \varepsilon p_1^{n_1} \cdot \ldots \cdot p_s^{n_s} \in H \ \text{for some $\varepsilon \in F^{\times}$} \} \subset (\N_0^s, +)
\]
is called the {\it value semigroup} of $H$. Clearly, $\mathsf v(H)$ is finitely primary again. Value semigroups play a crucial role in the study of one-dimensional local domains (see, e.g., \cite{Ba-An-Fr00a, Ba-Do-Fo97}).

\smallskip
The next lemma offers a further (but less straightforward) sufficient condition for local tameness. As the second part of the lemma indicates, such a condition ensures the local tameness of the seminormalization of any strongly primary monoid whose complete integral closure is Krull.

\smallskip
\begin{lemma} \label{2.1}
	Let $H$ be a strongly primary monoid.
	\begin{enumerate}
		\item If $(H \DP \widehat H) \ne \emptyset$, then $H$ is locally tame.
		\item \label{item:strongly primary monoids lemma 2} If $\widehat H$ is Krull, then  $H'$ is a locally tame primary Mori monoid with $\widehat{H'} = \widehat H$ and $(H' \DP \widehat H)\ne \emptyset$.
	\end{enumerate}
\end{lemma}

\begin{proof}
	1.  This part follows from \cite[Theorem~3.9]{Ge-Ro19b}.
	
	\smallskip
	2. Suppose that $\widehat H$ is Krull. Since $\widehat H$ is completely integrally closed, it follows that $\widehat{H'} = \widehat H$. The monoid $H'$ is seminormal, and it is primary by \cite[Theorem~15.4]{HK98} (we use that $H' = \widetilde H$ by \eqref{primary-closure}). Since the  conductor of a seminormal primary monoid is non-empty by \cite[Proposition~4.8]{G-HK-H-K03}, we infer that $(H' \DP \widehat H)\ne \emptyset$. Since $H'$ is primary and $\widehat H$ is Krull, $H'$ is Mori by \cite[Theorem~4.1]{Re13a}. Thus, $H'$ is strongly primary with non-empty conductor and hence it is locally tame by part~1.
\end{proof}

If $H$ is strongly primary, then $\widehat {H'}$ is completely integrally closed, and so is $\widehat H$  provided that $(H \DP \widehat H) \ne \emptyset$. If $H$ is a primary Mori monoid with $(H \DP \widehat H) \ne \emptyset$, then $\widehat H$ is Krull. However, in general, the complete integral closure of a primary Mori monoid (note even of a primary Mori domain) may not be Mori (\cite[Example~9]{Lu00}) or completely integrally closed (for a survey see \cite[Section~7]{Ba00}). The assumptions made in Lemma~\ref{2.1} are most natural for strongly primary monoids stemming from ring theory, whose properties are however in stark contrast to the analogous properties of strongly primary Puiseux monoids (as we will see in Theorem~\ref{thm:structure-of-strongly-primary-and-its-closure}).

We end this subsection with a characterization of global tameness, and for this we need to introduce the $\omega (H, \cdot)$ invariants.  For an atomic monoid $H$ and an atom $u \in \mathcal A (H)$, let
$\omega (H,u)$ be the smallest $N \in \N_0 \cup \{\infty\}$ with the following property:
\begin{itemize}
      \item[] If $k \in \N$ and $a_1, \dots, a_k \in H$ (it is equivalent to assume that $a_1, \dots, a_k \in \mathcal A (H)$) with $u \mid a_1 \cdot \ldots \cdot a_k$, then there is a subset $\Omega \subset \llb 1,k \rrb$ such that
      \[
     		|\Omega| \le N \quad \text{and} \quad u \mid \prod_{\lambda \in \Omega} a_{\lambda} \,.
      \]
\end{itemize}
We set $\omega (H) = \sup \{ \omega (H,u) : u \in \mathcal A (H) \} \in \N \cup \{\infty\}$.

\smallskip
\begin{lemma} \label{lem:global-tameness}
	Let $H$ be a strongly primary monoid. Then the following statements are equivalent.
	\begin{enumerate}
		\item[(a)] $H$ is globally tame.
		\item[(b)] $\bigcap_{u \in \mathcal A (H)} u H \ne \emptyset$.
		\item[(c)] $\mathcal M (H) < \infty$.
	\end{enumerate}
\end{lemma}

\begin{proof}
Parts (a) and (b) are equivalent by \cite[Theorem~3.8]{Ge-Ro19b}. Next we show that~(c) implies~(a). By definition, we obtain that $\omega (H,u) \le \mathcal M (u)$ for all $u \in \mathcal A(H)$, whence $\omega (H) \le \mathcal M (H)$. Since $\mathsf t (H) \le \omega (H)^2$ by \cite[Proposition 3.5]{Ge-Ka10a}, part~(a) holds.
To show that~(b) implies~(c) we choose an element $f$ in the given intersection. If $u \in \mathcal A(H)$, then $(H \setminus H^{\times})^{\mathcal M (f)} \subset fH \subset uH$. As a result, $\mathcal M(u) \le \mathcal M (f)$ for all $u \in \mathcal A(H)$, which implies that $\mathcal M(H) \le \mathcal M(f)$.
\end{proof}
\medskip

\section{Strongly primary Puiseux  monoids}
\label{sec:Puiseux monoids}

\smallskip
In this section we focus on the study of submonoids of $(\Q_{\ge 0}, +)$, here called {\it Puiseux monoids}. This gives rise to an entirely new class of strongly primary monoids with arithmetic properties that stand in stark contrast to those of primary domains (as outlined in Subsection~\ref{sec:background}.\ref{background-strongly-primary}). The main results in this section are Theorems~\ref{thm:structure-of-strongly-primary-and-its-closure},~\ref{thm:a SPPM has finite elasticity iff it is globally tame}, and \ref{thm:Puiseux-monoids-and-their-quotient-groups}. In particular, Theorem 3.3 characterizes numerical monoids as the only Puiseux monoids that are both finitely and strongly primary, and as the only Puiseux monoids that are both strongly primary and primary Mori with non-empty conductor.

\smallskip
\subsection{Algebraic closures and conductor of Puiseux monoids} \label{subsec:conductor} First, we introduce some useful notation. If $x \in \Q$, then the unique integers $a$ and $b$ such that $b \ge 1$, $x = a/b$, and $\gcd(a,b)=1$ are denoted by $\mathsf{n}(x)$ and $\mathsf{d}(x)$, and called the \textit{numerator} and the \textit{denominator} of $x$, respectively. For each $X \subset \Q$, we set $\mathsf{n}(X) = \{\mathsf{n}(x) : x \in X\}$ and $\mathsf{d}(X) = \{\mathsf{d}(x) : x \in X\}$. Moreover, we set $X^\bullet = X \setminus \{0\}$ and let $\langle X \rangle$ denote the submonoid of $(\Q,+)$ generated by $X$.

Let $H$ be a Puiseux monoid that is nontrivial (i.e., $H \neq \{0\}$). As $H \cap \N$ is non-empty, for $x,y \in H$ satisfying that $y \in x + H$ we write $x \mid_H y$ instead of $x \mid y$ to avoid conflicts with the standard notation of division in the multiplicative monoid $\N$.

Let $H$ be a nontrivial Puiseux monoid. For any $x,y \in H^\bullet$, it follows that $\mathsf{n}(x) \mathsf{d}(y)y = \mathsf{n}(y)\mathsf{d}(x) x \in x + H$. Hence $H$ is primary, and $H^{\bullet}$ is its maximal ideal.

\begin{proposition} \label{prop:closure-of-Puiseux-monoids}
	Let $H$ be a nontrivial Puiseux monoid, and set $n = \gcd  \mathsf n (H)$. Then the following statements hold.
	\begin{enumerate}
		\item \label{item:prop closure, equality of closures} $H' = \widetilde{H} = \widehat H = \mathsf q (H) \cap \Q_{\ge 0} = n \cdot  \langle 1/d : d \in \mathsf{d}(H^\bullet) \rangle$.
		\item \label{item 2:prop closure, equality of closures} $H$ is a valuation monoid if and only if $H$ is seminormal.
		\item \label{item 3:prop closure, the complete integral closure is a valuation monoid}$ \widehat{H}$ is a valuation monoid.
	\end{enumerate}
\end{proposition}
\begin{proof}

1.  The first equality follows from~\eqref{primary-closure} and the fact that Puiseux monoids are primary. Clearly, we have $H' \subset \widehat H \subset \mathsf q (H) \cap \Q_{\ge 0}$. In order to show that $\mathsf q (H) \cap \Q_{\ge 0} \subset n \cdot  \langle 1/d : d \in \mathsf{d}(H^\bullet) \rangle$, we choose an element $z=y-x \in \mathsf q (H) \cap \Q_{\ge 0}$ with $x,y \in H$. Then $\mathsf d (z)$ divides $\lcm ( \mathsf d (x), \mathsf d (y) )$, and a straightforward calculation shows that $\mathsf{d}(H^\bullet)$ is closed under taking positive divisors and least common multiples. Hence $\mathsf d (z) \in \mathsf d (H^{\bullet})$ and
\[
	z = \frac{\mathsf n (z)}{\mathsf d (z)} \in n \cdot  \langle 1/d : d \in \mathsf{d}(H^\bullet) \rangle.
\]
To verify that $n \cdot  \langle 1/d : d \in \mathsf{d}(H^\bullet) \rangle \subset H'$, take $k \in \N$ and $q_1, \dots, q_k \in H^\bullet$, and then consider
\[
	x = n \left( \frac{m_1}{\mathsf{d}(q_1)} + \dots + \frac{m_k}{\mathsf{d}(q_k)} \right)  \in n \cdot \langle 1/d : d \in \mathsf{d}(H^\bullet) \rangle,
\]
where $m_1, \dots, m_k \in \N$. For each $i \in \llb 1,k \rrb$, the fact that $\mathsf{n}(q_i)\big(\frac{m_i}{\mathsf{d}(q_i)} n \big) = m_i n q_i \in H$ implies that $\frac{m_i}{\mathsf{d}(q_i)} n \in \widetilde H$. Thus, $x \in \widetilde H = H'$ by~\eqref{primary-closure}.

\smallskip
2. It is well known that every valuation monoid is seminormal. Conversely, suppose that $H$ is seminormal. Then it follows from the previous part that $H = \mathsf q (H) \cap \Q_{\ge 0} $. Thus, $y - x \in \mathsf q (H) \cap \Q_{\ge 0} = H$ for all $x, y \in H$ with $x \le y$. Hence $H$ is a valuation monoid.

\smallskip
3. This is an immediate consequence of parts~1 and~2.
\end{proof}

With this in hand, we can give an explicit description of the conductor of a Puiseux monoid. We assume that $0$ is the supremum of the empty set.

\begin{proposition} \label{prop:conductor-of-Puiseux-monoids}
	Let $H$ be a Puiseux monoid with conductor $\mathfrak f = (H \DP \widehat{H})$, and set $\sigma = \sup \, ( \widehat{H} \setminus H )$. Then
	\[
		\mathfrak f =
		\begin{cases}
			\emptyset & \text{if } \sigma = \infty \,, \\
			H_{\ge \sigma}  & \textit{if }\sigma < \infty \,.
		\end{cases}
	\]
	In particular, if $H = \widehat H$, then $\mathfrak f = H_{\ge 0} = H$.
\end{proposition}

\begin{proof}
	Clearly, $H = \widehat H$ implies that $(H \DP \widehat H)= H = H_{\ge 0}$. Suppose $H \ne \widehat H$ and note that $(H : \widehat{H}) \subset {H}$. We distinguish two cases.

	\vskip 0.1cm
	\textsc{Case 1:} $\sigma = \infty$. Take $x \in \widehat{H}$. Since $\widehat{H} \setminus H$ is unbounded, there exists $x' \in \widehat{H} \setminus H$ such that $x' > x$. Then taking $y = x' - x \in \widehat{H}$, one can see that $x + y = x' \notin H$. As a result, $x \notin (H : \widehat{H})$. Thus, one obtains that $\mathfrak f = \emptyset$.
	\vskip 0.1cm
	\textsc{Case 2:} $\sigma < \infty$. Proceeding as we did in the previous paragraph, we can argue that $\widehat{H}_{< \sigma}$ and $(H : \widehat{H})$ are disjoint, which implies that $(H : \widehat{H}) \subset H_{\ge \sigma}$. To show the reverse inclusion, take $x \in H_{\ge \sigma}$. If $\sigma \notin H$, then $x > \sigma$ and so $x + \widehat{H} \subset (\widehat{H} + \widehat{H})_{> \sigma} \subset \widehat{H}_{> \sigma} = H_{> \sigma} \subset H$. So $x \in (H : \widehat{H})$ provided that $\sigma \notin H$. If $\sigma \in H$, then $\widehat{H}_{\ge \sigma} = H_{\ge \sigma}$ and so $x + \widehat{H} \subset (\widehat{H} + \widehat{H})_{\ge \sigma} \subset \widehat{H}_{\ge \sigma} \subset H$. Therefore $x \in (H : \widehat{H})$ also when $\sigma \in H$. Hence $H_{\ge \sigma} \subset (H : \widehat{H})$, yielding $\mathfrak f = H_{\ge \sigma}$.
\end{proof}

\smallskip
\subsection{Strongly primary Puiseux monoids} \label{subsec:strongly-primary} We start with a result demonstrating how exceptional strongly primary Puiseux monoids are inside the class of all strongly primary  monoids. 

\begin{theorem} \label{thm:structure-of-strongly-primary-and-its-closure}
	Let $H$ be a strongly primary Puiseux monoid.
	\begin{enumerate}
		\item The following statements are equivalent.
		\begin{enumerate}
			\item[(a)] $H$ is finitely generated.
			\item[(b)] $H$ is isomorphic to a numerical monoid.
			\item[(c)] $H$ is Mori with $(H \DP \widehat H) \ne \emptyset$.
			\item[(d)] $H$ is finitely primary.
			\item[(e)] $\widehat H$ is Krull.
			\item[(f)] $\widehat H$ is atomic.
			\item[(g)] $\widehat H$ is strongly primary.
			\item[(h)] $\widehat H \cong (\N_0, +)$.
		\end{enumerate}
		\smallskip
		\item \label{item:strongly primary necessary condition} The set $A_p = \{u \in \mathcal A(H): p \mid \mathsf d(u)\}$ is finite and non-empty only for finitely many $p \in \mathbb P$. In particular, $H$ is isomorphic to a Puiseux monoid $H^\ast$ with the additional property that, for every $p \in \mathbb P$, the set $A_p^\ast = \{u \in \mathcal A(H^\ast): p \mid \mathsf d(u)\}$ is either empty or infinite.
	\end{enumerate}
\end{theorem}

\begin{proof}
	1.	(a) $\Leftrightarrow$ (b) It follows from \cite[Proposition~3.2]{Go17a}.
		
		(b) $\Rightarrow$ (c) Finitely generated monoids are Mori with non-empty conductor by \cite[Theorem~2.7.13]{Ge-HK06a}.
		
		(c) $\Rightarrow$ (e) $\Rightarrow$ (f) The complete integral closure of a Mori monoid with non-empty conductor is Krull by \cite[Theorem~2.3.5]{Ge-HK06a}, and Krull monoids are atomic.
		
		(f) $\Rightarrow$ (h) Atomic valuation monoids are discrete valuation monoids by \cite[Theorem~16.4]{HK98}, whence $\widehat H \cong (\N_0, +)$.
		
		(h) $\Rightarrow$ (a) Since $\widehat H = H' = \widetilde H$ by Proposition~\ref{prop:closure-of-Puiseux-monoids}.\ref{item:prop closure, equality of closures}, it follows that $\widetilde H$ is finitely generated. Thus,~$H$ is finitely generated by \cite[Proposition~6.1]{Ge-HK94}.
		
		(h) $\Rightarrow$ (g) $\Rightarrow$ (f) Both implications are obvious.
	
		(b) $\Rightarrow$ (d) $\Rightarrow$ (e) Clearly, every numerical monoid is finitely primary. The complete integral closure of a finitely primary monoid is factorial and hence Krull (\cite[Theorem~2.9.2]{Ge-HK06a}).
	
	\smallskip
	2. Assume to the contrary that the set $\mathcal P = \{p \in \mathbb P : 1 \le |A_p| < \infty\}$ is infinite. Accordingly, let $(p_k)_{k \ge 1}$ be the unique enumeration of $\mathcal P$ with $p_1 < p_2 < \cdots,$ and for every $p \in \mathcal P$  set $t_p = \prod_{u \in A_p} \mathsf d(u)$. We recursively construct a strictly increasing sequence $(\kappa_n)_{n \ge 1}$ of positive integers, as follows. We start with $\kappa_1 = 1$. Then, for each $n \in \mathbb N$, we let
	\[
		\mathcal P_n = \{p \in \mathcal P: p \mid t_{p_1} \cdot \ldots \cdot t_{p_{\kappa_n}}\},
	\]
	and we let $\kappa_{n+1}$ be the smallest integer $k > \kappa_n$ such that $p_k$ does not divide $a_n = \prod_{p \in \mathcal P_n} t_p$.
	
	Now, take $m, n \in \mathbb N$ with $m < n$. We claim that
	\begin{equation}\label{equ:well-separated(0)}
		\{p \in \mathcal P: p \mid \gcd(\mathsf d(u), \mathsf d(v))\} = \emptyset
	\end{equation}
	for every $u \in A_{p_{\kappa_m}}$ and $v \in A_{p_{\kappa_n}}$\!. If not, there exists $p \in \mathcal P$ such that $p \mid \mathsf d(u)$ and $p \mid \mathsf d(v)$ for some $u \in A_{p_{\kappa_m}}$ and $v \in A_{p_{\kappa_n}}$. Then $v \in A_p$, and this in turn implies that $p_{\kappa_n} \mid t_p$, because $p_{\kappa_n} \mid \mathsf d(v)$. On the other hand, $p \in \mathcal P_{n-1}$, since $p \mid t_{\kappa_m}$ and $\kappa_m \le \kappa_{n-1}$ (by the fact that $m < n$ and the sequence $(\kappa_n)_{n \ge 1}$ is strictly increasing). As a result, we conclude that $p_{\kappa_n} \mid a_{n-1}$, contradicting that $p_{\kappa_n}$ is not a prime divisor of $a_{n-1}$ by construction.
	
	Building on these premises, we choose some  $q \in H \cap \mathbb N$ and set $u_n = \min A_{p_{\kappa_n}}$ for each $n \in \mathbb N$. Since~$H$ is strongly primary, there exists $m \in \N$ such that $m H^{\bullet} \subset q + H$. In particular,
	\[
		u_1 + \dots + u_m \in q + H.
	\]
	Because $H$ is atomic, there exist $s \in \mathbb N$, $v_1, \dots, v_s \in \mathcal A(H)$, and $\alpha_1, \dots, \alpha_s \in \mathbb N_0$ with
	\begin{equation}
		\label{equ:strongly-primary-condition(0)}
		\sum_{i=1}^m u_i = q + \sum_{i=1}^s \alpha_i v_i.
	\end{equation}
	From \eqref{equ:well-separated(0)}, there is no prime $p \in \mathcal P$ dividing both $\mathsf d(u_i)$ and $\mathsf d(u_j)$ for some $i, j \in \mathbb N$ with $i \ne j$. Therefore we see that, for every $i \in \llb 1, m \rrb$, there is an index $r_i \in \llb 1, s \rrb$ such that $\alpha_{r_i} \ne 0$ and $p_{k_i} \mid \mathsf d(v_{r_i})$; otherwise the $p_{k_i}$-valuation of the left-hand side of \eqref{equ:strongly-primary-condition(0)} would be different from the one of the right-hand side.
	
	On the other hand, it also follows from \eqref{equ:well-separated(0)} that there is no atom $v \in \mathcal A(H)$ such that $p_{k_i} \mid \mathsf d(v)$ and $p_{k_j} \mid \mathsf d(v)$ for some $i,j \in \mathbb N$ with $i \ne j$. Consequently, it is clear that $r_i \ne r_j$ for all $i, j \in \llb 1, m \rrb$ with $i \ne j$.
	
	Then using that $u_n = \min A_{p_{\kappa_n}}$ for every $n \in \mathbb N$, we obtain
	\[
		\sum_{i=1}^m u_i \ge q + \sum_{i=1}^m \alpha_{r_i} v_{r_i} \ge q + \sum_{i=1}^m v_{r_i} > \sum_{i=1}^m v_{r_i} \ge \sum_{i=1}^m u_i \,,
	\]
	a contradiction.
	Thus, the set $\mathcal P = \{p \in \mathbb P: 1 \le |A_p| < \infty\}$ is finite. Therefore $q = \prod_{p \in \mathcal P} \prod_{u \in A_p} \mathsf d(u) \in \N$, and so~$H$ is isomorphic to the Puiseux monoid $H^\ast = q \cdot H$, which has the required property.
\end{proof}

We proceed to offer various characterizations of strongly primary Puiseux monoids with non-empty conductor.

\begin{theorem} \label{thm:strongly primary characterization when conductor non-empty}
	Let $H$ be a nontrivial Puiseux monoid. Each of the first three statements below implies its successor. Moreover, if $(H : \widehat{H}) \neq \emptyset$, then all four statements are equivalent.
	\begin{enumerate}
		\item[(a)] $H$ is strongly primary.
		\item[(b)] $0$ is not a limit point of $H^\bullet$\!.
		\item[(c)] $H$ is a \BF-monoid.
		\item[(d)] $H$ satisfies the \text{\rm ACCP} (ascending chain condition on principal ideals).
	\end{enumerate}
\end{theorem}

\begin{proof}
	(a) $\Rightarrow (b)$ Take $x \in H^{\bullet}$. Since $H$ is strongly primary, there exists $m \in \N$ such that $m H^{\bullet} \subset x + H$. This implies that $\inf H^\bullet \ge x/m > 0$, and hence $0$ is not a limit point of $H^\bullet$.
	
	(b) $\Rightarrow$ (c) It follows from \cite[Proposition 4.5]{Go19b}.
	
	(c) $\Rightarrow$ (d) It follows from \cite[Corollary 1.3.3]{Ge-HK06a}.
	
	For the rest of the proof, suppose that $(H : \widehat{H}) \neq \emptyset$.
	
	(d) $\Rightarrow$ (b) Suppose, by way of contradiction, that $0$ is a limit point of $H^\bullet$. Since $H$ satisfies the ACCP, it is atomic. Then there exists a sequence of atoms $(u_i)_{i \ge 1}$ such that $u_i < 1/2^i$ for every $i \in \N$. Then the series $\sum_{n=1}^\infty u_n$ converges to a limit $ \ell \in ]0,1[$ and, as a consequence, $s_i := \sum_{n=1}^i u_n \in H \cap \, ]0,1[$ for every $i \in \N$. Since $(H : \widehat{H}) \neq \emptyset$, Proposition~\ref{prop:conductor-of-Puiseux-monoids} guarantees that $\widehat{H} \setminus H$ is bounded. Take $x \in H$ such that $x > 1 + \sup \widehat{H} \setminus H$. It follows from Proposition~\ref{prop:closure-of-Puiseux-monoids}.\ref{item:prop closure, equality of closures} that $x - s_i \in \widehat{H}$ for every $i \in \N$. As $(H : \widehat{H}) \neq \emptyset$ and $x - s_i > \sup \widehat{H} \setminus H$ for each $i \in \N$, Proposition~\ref{prop:conductor-of-Puiseux-monoids} ensures that $x - s_i \in H$ for each $i \in \N$. Consider the sequence of principal ideals $(x - s_i + H )_{i \ge 1}$ of~$H$. Since $x - s_i = x - s_{i+1} + (s_{i+1} - s_i) = (x - s_{i+1}) + u_{i+1}$ for every $i \in \N$, the sequence $( x - s_i + H )_{i \ge 1}$ is an ascending chain of principal ideals. Because each $u_i$ is positive, the chain of ideals $( x - s_i + H )_{i \ge 1}$ does not stabilize. However, this contradicts that $H$ satisfies the ACCP.
	
	(b) $\Rightarrow (a)$ Take $\varepsilon \in \,\, ]0, \inf H^\bullet[$. As $(H : \widehat{H}) \neq \emptyset$, Proposition~\ref{prop:conductor-of-Puiseux-monoids} implies that $\sigma = \sup \widehat H \setminus H < \infty$. Take $a \in H^{\bullet}$, and set $n = \lceil \varepsilon^{-1}(a + \sigma) \rceil$. We claim that the $n$-fold sumset $n H^{\bullet}$ is contained in $a +H$. Indeed, if $x_1, \dots, x_n \in H^{\bullet}$, then $x = x_1 + \dots + x_n > n \varepsilon \ge a + \sigma$. Now Proposition~\ref{prop:closure-of-Puiseux-monoids}.\ref{item:prop closure, equality of closures} guarantees that $x-a \in \mathsf q (H) \cap \Q_{> \sigma} = \widehat H_{> \sigma} = H_{> \sigma}$. As a result, $x \in a + H_{> \sigma} \subset a + H$.
\end{proof}

Without the non-empty-conductor condition, none of the last three statements in Theorem~\ref{thm:strongly primary characterization when conductor non-empty} implies its predecessor, as the following example indicates.

\begin{example} \hfill \label{ex:when empty conductor SP/dense/BF/ACCP are not equivalent for PMs}
	
	(d) $\not \Rightarrow$ (c) The monoid $H = \langle 1/p : p \in \P \rangle$ satisfies the ACCP by \cite[Theorem~5.2]{Go19c}, and it is\ not hard to verify that $\mathcal{A}(H) = \{ 1/p : p \in \P \}$. However, $H$ is not a BF-monoid because $p \in \mathsf{L}(1)$ for every $p \in \P$.
	
	\smallskip
	(c) $\not \Rightarrow$ (b) Let $(p_i)_{i \ge 1}$ and $(q_i)_{i \ge 1}$ be two strictly increasing sequences consisting of primes such that $q_i > p_i^2$ for every $i \in \N$. Set $H = \langle p_i/q_i : i \in \N \rangle$. It follows from~\cite[Corollary~5.6]{Go-Go18} that $H$ is atomic, and one can easily check that $\mathcal{A}(H) = \{p_i/q_i : i \in \N\}$. To argue that $H$ is a BF-monoid, take $x \in H^\bullet$ and note that since both sequences $(p_i)_{i \ge 1}$ and $(q_i)_{i \ge 1}$ are strictly increasing, there exists $N \in \N$ such that $q_i \nmid \mathsf{d}(x)$ and $p_i > x$ for every $i \ge N$. As a result, if $z \in \mathsf{Z}(x)$, then none of the atoms in $\{p_i/q_i : i \ge N\}$ can appear in $z$. Therefore $|\mathsf{L}(x)| \le |\mathsf{Z}(x)| < \infty$. Hence $H$ is a BF-monoid. However, $q_i > p_i^2$ for every $i \in \N$ implies that $0$ is a limit point of $H^\bullet$.
	
	\smallskip
	(b) $\not \Rightarrow$ (a) For $r \in \Q_{> 1} \setminus \N$, consider the Puiseux monoid $H = \langle r^i : i \in \N_0 \rangle$. It follows from \cite[Theorem~5.6]{Go19b} that $H$ is an FF-monoid and, therefore, atomic. In addition, it follows from \cite[Theorem~6.2]{Go-Go18} that $\mathcal{A}(H) = \{r^i : i \in \N_0\}$. Suppose, by way of contradiction, that $H$ is strongly primary. Then there exists $n \in \N$ such that $n H^\bullet \subset 1 + H$. Consider the element $x = r + r^2 + \dots + r^n \in n H^\bullet$. Now \cite[Lemma~3.2]{Ch-Go-Go20} guarantees that $|\mathsf{Z}(x)| = 1$. Thus, $1 \nmid_H x$, which contradicts that $n H^\bullet \subset 1 + H$. Hence~$H$ is not strongly primary even though~$0$ is not a limit point of $H^\bullet$.
\end{example}
\smallskip

\begin{remark}
	In \cite{fG18}, the second author defines a monoid $H$ to be ``strongly primary" if it is a reduced primary BF-monoid for which there exist $n \in \N$ and a finite subset $S$ of $H^\bullet$ with $\{nh : h \in H^\bullet\} \subset S + H$. This definition is not equivalent to the definition of strongly primary monoid used in this paper, which is the most standard in the literature. The definition we use here is indeed stronger than that one used in \cite{fG18}. For instance, we have just verified in the last paragraph of Example~\ref{ex:when empty conductor SP/dense/BF/ACCP are not equivalent for PMs} that for every $r \in \Q_{>1} \setminus \N$ the Puiseux monoid $H = \langle r^i : i \in \N_0 \rangle$ is not strongly primary, even though $H$ is strongly primary with respect to the definition used in~\cite{fG18} (see \cite[Proposition~5.6]{fG18}).
\end{remark}

Although being a BF-monoid and satisfying the ACCP are equivalent conditions for a Puiseux monoid with non-empty conductor, there are BF-monoids in this class that are not FF-monoids as well as atomic monoids with non-empty conductors which do not belong to this class. The next example illustrates this observation.

\begin{example} \label{ex:two tests for characterization theorem of SPPM with nonempty conductor}
	First, consider the Puiseux monoid $H = \{0\} \cup \Q_{> 1}$. It follows from Proposition~\ref{prop:conductor-of-Puiseux-monoids} that $(H : \widehat{H}) \neq \emptyset$. In addition, Theorem~\ref{thm:strongly primary characterization when conductor non-empty} guarantees that $H$ is a BF-monoid. Note that $\mathcal{A}(H) = H \cap \, ]1,2]$. Now observe that for all $x \in H \cap \, ]2,3]$ the expression $(1 + 1/n) + (x - 1 - 1/n)$ is a length-$2$ factorization in $\mathsf{Z}(x)$ for every integer $n > \frac{1}{x-2}$. This implies that $|\mathsf{Z}(x)| = \infty$ for all $x \in H_{> 2}$ and, therefore, $H$ is not an FF-monoid.
	
	\smallskip
	Now consider the Puiseux monoid $H =  \langle 1/p : p \in \P \rangle \cup \Q_{\ge 1}$. Since the monoid $\langle 1/p : p \in \P \rangle$ is atomic by \cite[Theorem~5.5]{Go-Go18}, it is not hard to check that $H$ is also atomic. It follows from Proposition~\ref{prop:conductor-of-Puiseux-monoids} that $(H : \widehat{H}) \neq \emptyset$. Because~$0$ is a limit point of~$H^\bullet$, Theorem~\ref{thm:strongly primary characterization when conductor non-empty} ensures that $H$ is neither strongly primary nor a monoid satisfying the ACCP.
\end{example}

\smallskip
\subsection{Characterization of global tameness for strongly primary Puiseux monoids} Before proceeding to characterize  strongly primary Puiseux monoids that are globally tame, let us collect the following lemma.

\begin{lemma} \label{lem:elasticity-of-Puiseux-monoids}
	Let $H$ be a strongly primary Puiseux monoid. Then $\rho (H) = \sup \mathcal A (H)/ \inf \mathcal A (H)$. In particular, $H$ has accepted elasticity if and only if the supremum and the infimum of $\mathcal A (H)$ are attained, and $\rho (H)=1$ if and only if $H \cong (\N_0, +)$.
\end{lemma}

\begin{proof}
	As $H$ is strongly primary, $0$ is not a limit point of $H^\bullet$ and, therefore, the formula for $\rho (H)$ follows from \cite[Theorem~3.2]{Go-ON19}. This implies the statement on accepted elasticity. If $\rho (H)=1$, then $\sup \mathcal A (H) = \inf \mathcal A (H)$, which means that $|\mathcal A (H)|=1$. In this case, $H$ is isomorphic to $(\N_0, +)$. Conversely, if $H \cong (\N_0, +)$, then $H$ is factorial and $\rho (H)=1$.
\end{proof}

The following characterization of strongly primary globally tame Puiseux monoids should be compared with the corresponding results for finitely primary monoids (Theorem~\ref{global-tameness-finitely-primary}) and for strongly primary domains (Theorem~\ref{global-tameness-domains}). For these classes of strongly primary monoids, global tameness is equivalent to the finiteness of the elasticity. We should also notice that this equivalence holds true for wide classes of Noetherian domains \cite[Theorems~6.2 and~7.2]{Ka16b} but not for all Krull monoids \cite[Proposition~4.1 and Theorem~4.2]{Ge-Gr-Sc-Sc10}). However, neither for finitely primary monoids nor for strongly primary domains, the non-emptiness of the conductor implies global tameness. Conversely, it is easy to see that a globally tame monoid might have empty conductor. Indeed, consider a finitely generated monoid $H_0$ that is not Krull. Then $H_0$ is globally tame and $(H_0 : \widehat{H_0}) \ne \emptyset$. As a consequence, the monoid $H := \coprod_{i \ge 1} H_i$, where $H_i = H_0$ for all $i \in \N$, is globally tame with $\mathsf t(H) = \mathsf t (H_0)$ and $(H : \widehat H) = \emptyset$.

\begin{theorem} \label{thm:a SPPM has finite elasticity iff it is globally tame}
	Let $H$ be a strongly primary Puiseux monoid.
	\begin{enumerate}
		\item The following statements are equivalent.
		\begin{enumerate}
			\item[(a)] $H$ is globally tame.
			\item[(b)] $\rho(H) < \infty$.
			\item[(c)] $(H : \widehat{H}) \neq \emptyset$.
		\end{enumerate}
	    If these statements hold, then $\Lambda (H)=\infty$.
		\smallskip
		\item \label{item: strongly primary globally tame part 2} If $\Lambda (H) < \infty$, then $H$ is locally tame but not globally tame.
	\end{enumerate}
\end{theorem}

\begin{proof}
	1.	(a) $\Rightarrow$ (b) This holds true for all atomic monoids (see \eqref{elasticity-tameness}).
	
	(b) $\Rightarrow$ (c) Because $H$ has finite elasticity, Lemma~\ref{lem:elasticity-of-Puiseux-monoids} guarantees that $\mathcal{A}(H)$ is bounded. Now take $\alpha = \sup \mathcal{A}(H)$ and fix $u \in \mathcal{A}(H)$. As $H$ is strongly primary, $\mathcal{M}(u) H^\bullet \subset u + H$. Consider the set
	\[
		S := \{q \in \mathsf{q}(H) : q \notin H \ \text{and} \ q \ge \mathcal{M}(u)\alpha\}.
	\]
	We claim that $S$ is the empty set. To prove this, take $q \in \mathsf{q}(H)$ satisfying that $q \ge \mathcal{M}(u)\alpha$, and then write $q = x-y$ for some $x,y \in H$ with $x > y > 0$. Since $x +(\mathcal{M}(y) -  1)u \in \mathcal{M}(y)H^\bullet \subset y + H$, one obtains that $x + (\mathcal{M}(y) -  1)u - y \in H$. Consequently,
	\[
		q = x-y = ( x + (\mathcal{M}(y)- 1)u - y ) - (\mathcal{M}(y) - 1)u \in H - \langle u \rangle,
	\]
	where $\langle u \rangle$ is the cyclic monoid generated by $u$. Now set
	\[
		n_0 := \min \{ n \in \N_0 : q = x - nu \text{ for some } x \in H \},
	\]
	and $x_0 := q + n_0 u$ (i.e., $q = x_0 - n_0 u$). Suppose by contradiction that $n_0 > 0$. Take $m \in \N$ and $u_1, \dots, u_m \in \mathcal{A}(H)$ such that $x_0 = u_1 + \dots + u_m$. Notice that
	\[
		m \alpha \ge m \max \{ u_i : i \in \llb 1,m \rrb \} \ge x_0 > q \ge \mathcal{M}(u) \alpha.
	\]
	Therefore $m > \mathcal{M}(u)$ and so
	\[
		x_0 = u_1 + \dots + u_{\mathcal{M}(u) - 1} + \sum_{i= \mathcal{M}(u)}^m u_i \in \mathcal{M}(u) H^\bullet \subset u + H.
	\]
	So one can write $q = x_0' - (n_0 - 1)u$, where $x_0' = x_0 - u \in H$. However, this contradicts the minimality of $n_0$. As a result, $n_0 = 0$ and so $q = x_0 \in H$. This implies that $S$ is the empty set, as desired. Then $\widehat{H} \setminus H \subset \mathsf{q}(H) \setminus H \subset \Q_{\le \mathcal{M}(u)\alpha}$. Hence $\sup (\widehat{H} \setminus H) < \infty$, and $(H : \widehat{H}) \neq \emptyset$ by Proposition~\ref{prop:conductor-of-Puiseux-monoids}.
	
	(c) $\Rightarrow$ (a) If $(H\DP \widehat H) \ne \emptyset$, then $H$ is globally tame if and only if $\widehat H$ is a primary valuation monoid by \cite[Theorem~3.8]{Ge-Ro19b}. Hence the assertion follows from Proposition~\ref{prop:closure-of-Puiseux-monoids}.\ref{item 3:prop closure, the complete integral closure is a valuation monoid}. 

For the last statement of part~1, assume that (a), (b), and (c) hold. Now suppose for a contradiction that $\Lambda (H) < \infty$. Then $\rho_{\Lambda (H)} (H)= \infty$, whence $\rho_k (H)= \infty$ for all $k \ge \Lambda (H)$. Thus, $\rho(H)=\infty$, which contradicts statement~(b).
	
	\smallskip
	2. Suppose that $\Lambda (H) < \infty$. Then $H$ is locally tame by~\eqref{enforcing-local-tameness} and it is not globally tame  by part~1.
\end{proof}

As mentioned in the proof of Theorem~\ref{thm:a SPPM has finite elasticity iff it is globally tame}, every globally tame atomic monoid has finite elasticity and, therefore, the implication (a) $\Rightarrow$ (b) in Theorem~\ref{thm:a SPPM has finite elasticity iff it is globally tame} holds for all atomic Puiseux monoids. On the other hand, it is unknown whether the implication (a) $\Rightarrow$ (c) in the same theorem holds for all atomic Puiseux monoids. However, the strongly primary condition is crucial to establish the rest of the implications, as the following examples illustrates.

\begin{example} \hfill \label{ex:without strongly primaryness globally-tame/finite-elasticity/nonempty-conductor are not equivalent for PMs}
	
	(b) $\not \Rightarrow$ [(a) or (c)] Consider the Puiseux monoid $H$ generated by the set $\{1\} \cup \{ 1 + 1/p : p \in \P \}$. One can verify without  much effort that $\mathcal A(H) = \{1\} \cup \{ 1 + 1/p : p \in \P \}$ and, therefore, $H$ is atomic. Since $\mathcal{A}(H) \subset [1,2]$, it follows from Lemma~\ref{lem:elasticity-of-Puiseux-monoids} that $\rho(H) < \infty$, which is condition~(b). To check that $H$ is not globally tame, it suffices to argue that $\mathsf{t}(1) = \infty$. For $p \in \P$ consider the length-$p$ factorization $z_p = p (1 + 1/p) \in \mathsf{Z}(1+p)$. Note that every strict sub-factorization of $z_p$ produces an element of $H$ with unique factorization. This, along with the fact that $1 \mid_H 1+p$, shows that $\mathsf{t}(1) \ge p$. Hence $\mathsf{t}(1) = \infty$, and so~$H$ does not satisfy condition~(a). As $H$ satisfies~(b) but not~(a), it follows from Theorem~\ref{thm:a SPPM has finite elasticity iff it is globally tame} that $H$ is not strongly primary (it also follows from Theorem~\ref{thm:structure-of-strongly-primary-and-its-closure}.\ref{item:strongly primary necessary condition}). Since $0$ is not a limit point of $H$, Theorem~\ref{thm:strongly primary characterization when conductor non-empty} guarantees that $(H : \widehat{H}) = \emptyset$. Thus, $H$ does not satisfy condition~(c).
	
	\smallskip
	(c) $\not \Rightarrow$ [(a) or (b)] It is enough to show that (c) $\not \Rightarrow$ (b). To this end, consider the Puiseux monoid $H = \langle 1/p : p \in \P \rangle \cup \Q_{\ge 1}$ introduced in the second part of Example~\ref{ex:two tests for characterization theorem of SPPM with nonempty conductor}. We have already seen that $H$ is an atomic monoid satisfying that $(H : \widehat{H}) \neq \emptyset$; that is, $H$ satisfies condition~(c). On the other hand, since $0$ is a limit point of $H^\bullet$, it follows from~\cite[Theorem~3.2]{Go-ON19} that $\rho(H) = \infty$. So $H$ does not satisfy condition~(b).
\end{example}

\smallskip
\subsection{Explicit construction of strongly primary Puiseux monoids} The next example yields a large family of strongly primary Puiseux monoids depending on countably many parameters that can be conveniently tuned to illustrate a variety of specific phenomena.  Moreover, we will need these Puiseux monoids in the proof of Theorem~\ref{thm:Puiseux-monoids-and-their-quotient-groups} and Proposition~\ref{prop:strongly primary PM that is not locally tame}.

\begin{example}\label{exa:sophisticate}
	Let $(\alpha_i)_{i \ge 1}$ be a sequence of positive real numbers with $\alpha = \inf \{ \alpha_i : i \in \N \} > 0$, and let $(b_i)_{i \ge 1}$ be a sequence of positive integers such that $b_i \mid b_{i+1}$ for every $i \in \mathbb N$. Then, we define
	\begin{equation}\label{equ:def-of-complicate-monoid}
		H = \bigcup_{i \ge 1} H_i \subset \mathbb Q_{\ge 0},
		\quad\text{where}\quad
		H_i = \{0\} \cup (\mathbb Q_{\ge \alpha_i} \cap b_i^{-1} \cdot \mathbb Z) \subset \mathbb Q_{\ge 0}.
	\end{equation}
	
	\begin{claim} \label{claim:sophisticate(1)}
		Take $i, j \in \mathbb N$ with $i \le j$. Then the following statements hold.
		\begin{enumerate}
			\item \label{item:quotient group 1} $H_i^\bullet + H_j^\bullet \subset H_j^\bullet$ and $H_i^\bullet - H_j^\bullet \subset b_j^{-1} \cdot \mathbb Z$.
			\item \label{item:quotient group 2}$H$ is a Puiseux monoid and $\mathsf q(H) = \bigcup_{i \ge 1} b_i^{-1} \cdot \mathbb Z$.
		\end{enumerate}
	\end{claim}

	\begin{proof}
		1. Pick $x_i \in H_i^\bullet$ and $x_j \in H_j^\bullet$. Then $x_j \ge \alpha_j$, and there exist $a_i, a_j \in \mathbb N$ such that $x_i = a_i/b_i$ and $x_j = a_j/b_j$. So, taking into account that $b_i \mid b_j$, we obtain
		\[
			x_i + x_j \ge \alpha_j
			\quad\text{and}\quad
			x_i \pm x_j = \frac{a_i b_j/b_i \pm a_j}{b_j} \in b_j^{-1} \cdot \mathbb Z.
		\]
		In particular, we see that $x_i + x_j \in H_j^\bullet$.
		
		\smallskip
		2. It follows immediately from part~1 that $H$ is a Puiseux monoid. Also, it is clear that $\bigcup_{i \ge 1} b_i^{-1} \cdot \mathbb Z$ is a subgroup of $(\mathbb Q, +)$ containing $H$. In addition, Proposition~\ref{prop:closure-of-Puiseux-monoids}.\ref{item:prop closure, equality of closures} guarantees that $\bigcup_{i \ge 1} b_i^{-1} \cdot \mathbb Z$ is a subgroup of $\mathsf q(H)$ (note that $\gcd \mathsf n(H) = 1$ because $a/b_1$ and $(a+1)/b_1$ for all large $a \in \mathbb N$). Hence we are done, since $\mathsf q(H)$ is the smallest subgroup of $(\mathbb Q, +)$ containing $H$.
	\end{proof}

	\begin{claim} \label{claim:sophisticate(2)}
	 	The following statements hold.
	 	\begin{enumerate}
	 		\item $H$ is a BF-monoid and, in particular, an atomic monoid.
	 		
	 		\item \label{item:atomicity 2}$H_i \cap \mathcal{A}(H) \subset [\alpha_i, 2 \alpha_i + 1[$ for every $i \in \N$.
	 		
	 		\item \label{item:atomicity 3} $\mathcal A(H) \subset \bigcup_{i \ge 1} \bigl(H_i \cap {[\alpha_i, 2\alpha_i + 1[}\bigr)$.
	 	\end{enumerate}
	\end{claim}
	
	\begin{proof}
		1. Since $\inf H^\bullet = \alpha > 0$, it follows from Theorem~\ref{thm:strongly primary characterization when conductor non-empty} that $H$ is a BF-monoid. In particular, $H$ is atomic.
		
		\smallskip
		2. Take $u \in H_i \cap \mathcal{A}(H)$. Then $u \ge \alpha_i$ and $u = c/b_i$ for some $c \in \mathbb N$. Suppose, by way of contradiction, that $u \ge 2\alpha_i + 1$. In this case, $c \ge 2\alpha_i b_i + 1$, and one can see that
		\[
			\alpha_i \le \frac{\lceil \alpha_i b_i \rceil}{b_i} = u^\prime \in H_i^\bullet
			\quad\text{and}\quad
			\alpha_i = \frac{2\alpha_i b_i + 1 - (\alpha_i b_i + 1)}{b_i} < \frac{c - \lceil \alpha_i b_i \rceil}{b_i} = u - u^\prime \in  H_i^\bullet.
		\]
		This yields $u = u^\prime + (u - u^\prime) \in 2H^\bullet$. Hence $u \notin \mathcal A(H)$, which is a contradiction.
		
		\smallskip
		3. As an immediate consequence of part~2, one obtains that
		\[
			\mathcal{A}(H) = \bigcup_{i \ge 1} \big( H_i \cap \mathcal{A}(H) \big) \subset \bigcup_{i \ge 1} \big( H_i \cup [ \alpha_i, 2 \alpha_i + 1 [ \big).
		\]
	\end{proof}

	\begin{claim} \label{claim:sophisticate(3)}
		The monoid $H$ is strongly primary, and $\mathcal M(H) \le 1 + \sup \{ \lceil \alpha^{-1} (3\alpha_i + 1)\rceil : i \in \N \}$. In addition, the following statements are equivalent.
		\begin{enumerate}
			\item[(a)] $\sup \mathcal A(H) < \infty$.
			\item[(b)] $H$ is globally tame.
			\item[(c)] $(H : \widehat{H}) \neq \emptyset$.
			\item[(d)] $\rho(H) < \infty$.
			\item[(e)] $\rho_k(H) < \infty$ for every $k \in \N$.
			\item[(f)] $\lim_{i \to \infty} b_i < \infty$ or $\liminf \{\alpha_i : i \in \N\} < \infty$.
		\end{enumerate}
	\end{claim}
	
	\begin{proof}
	 	In order to show that $H$ is a strongly primary monoid, it suffices to verify that  $\mathcal M(u) < \infty$ for every $u \in \mathcal A(H)$. Take $u_0 \in \mathcal A(H)$. By Claim~\ref{claim:sophisticate(2)}.\ref{item:atomicity 3}, there exists $i_0 \in \mathbb N$ such that $u_0 \in H_{i_0}\cap [\alpha_{i_0}, 2\alpha_{i_0} + 1[$. Accordingly, set $n := \lceil 1 + \alpha^{-1} (3\alpha_{i_0} + 1) \rceil$. It is clear that $n \le 1 + \sup \{ \lceil \alpha^{-1} (3\alpha_i + 1)\rceil : i \in \N \}$. We aim to prove that $\mathcal M(u_0) \le n$. This will imply that $H$ is strongly primary. Consider $n$ elements $u_1, \ldots, u_n \in H^\bullet$, and set $u = u_1 + \dots + u_n - u_0$. For each $k \in \llb 1, n \rrb$ take $i_k \in \mathbb N$ such that $u_k \in H_{i_k}$. We may assume without loss of generality that $i_1 \ge \dots \ge i_n$. Then
		\[
			u \ge u_1 + (n-1)\alpha - u_0 \ge u_1 + (3\alpha_{i_0} + 1) - (2\alpha_{i_0} + 1) \ge \alpha_{i_1} + \alpha_{i_0} \ge \alpha_{\hat\imath},
		\]
		where $\hat\imath = \max(i_0, i_1)$. On the other hand, it is immediate from Claim~\ref{claim:sophisticate(1)}.\ref{item:quotient group 1} that
		\[
			u \in H_{i_1}^\bullet + \dots + H_{i_n}^\bullet - H_{i_0}^\bullet \subset H_{i_1}^\bullet - H_{i_0}^\bullet \subset b_{\hat\imath}^{-1} \cdot \mathbb Z.
		\]
		Consequently, $u \in \mathbb Q_{\ge \alpha_{\hat\imath}} \cap b_{\hat\imath}^{-1} \cdot \mathbb Z = H_{\hat\imath}^\bullet \subset H$.
		
		Let us proceed to show that the given conditions are equivalent.
		
		(a) $\Leftrightarrow$ (d) Since $\inf H^\bullet = \alpha > 0$, it follows from Lemma~\ref{lem:elasticity-of-Puiseux-monoids} that (a) and (d) are equivalent.
		
		(b) $\Leftrightarrow$ (c) $\Leftrightarrow$ (d) As $H$ is strongly primary, both equivalences follow from Theorem~\ref{thm:a SPPM has finite elasticity iff it is globally tame}.
		
		(a) $\Rightarrow (f)$ Suppose by contradiction that $\lim_{i \to \infty} b_i = \infty$ and $\liminf \{\alpha_i : i \in \N\} = \infty$. Because $H$ is atomic and $(b_i)_{i \ge 1}$ is an unbounded sequence, $|\mathcal{A}(H)| = \infty$. Fix $N \in \N$. As $\liminf \{\alpha_i : i \in \N\} = \infty$, the set $J = \{j \in \N : \alpha_j \le N\}$ is finite. In addition, $\mathcal{A}(H) \cap H_j$ is finite for every $j \in J$. Then Claim~\ref{claim:sophisticate(2)}.\ref{item:atomicity 3}, along with $|\mathcal{A}(H)| = \infty$, ensures that $\sup \mathcal{A}(H) \ge N$. Hence $\sup \mathcal{A}(H) = \infty$, which is a contradiction.
		
		$(f) \Rightarrow (a)$ We first observe that if $\lim_{i \to \infty} b_i < \infty$, then $H$ would be finitely generated, and so $\sup \mathcal{A}(H) < \infty$.
		On the other hand, suppose that $\ell := \liminf \{ \alpha_i : i \in \N\} < \infty$. Take a subsequence $(\alpha_{k_i})_{i \ge 1}$ of $(\alpha_i)_{i \ge 1}$ converging to $\ell$, and set $s = \sup\{ \alpha_{k_i} : i \in \N \}$. Then for every $j \in \N$ with $\ell < \alpha_j$, there exists $i \in \N$ such that $\alpha_{k_i} < \alpha_j$ and $b_j \mid b_{k_i}$. This implies that $H_j \subset H_{k_i}$, and so
		\[
			\mathcal{A}(H) \cap H_j \subset \mathcal{A}(H) \cap H_{k_i} \subset [ \alpha_{k_i}, 2 \alpha_{k_i} + 1 [ \, \subset [\inf H^\bullet, 2s + 1],
		\]
		where the second inclusion holds by Claim~\ref{claim:sophisticate(2)}.\ref{item:atomicity 2}. On the other hand, for those indices $j \in \N$ with $\ell \ge \alpha_j$,
		\[
			\mathcal{A}(H) \cap H_j \subset [ \alpha_j, 2\alpha_j + 1 [ \, \subset [\inf H^\bullet, 2 \ell + 1] \subset [\inf H^\bullet, 2s+1].
		\]
		Hence it follows from Claim~\ref{claim:sophisticate(2)}.\ref{item:atomicity 3} that $\sup \mathcal{A}(H) \le 2s+1 < \infty$.
		
		(d) $\Rightarrow$ (e) This holds for all atomic monoids via~\eqref{eq:elasticity in terms of local elasticities}.
		
		(e) $\Rightarrow$ (f) Suppose, by way of contradiction, that $\lim_{i \to \infty} b_i = \infty$ and $\liminf \{\alpha_i : i \in \N\} = \infty$. As $\liminf \{\alpha_i : i \in \N\} = \infty$ the set $S := \{ n \in \N : \alpha_n < \alpha_j \text{ for every } j \ge n \}$ has infinite cardinality. Taking $(k_i)_{i \ge 1}$ to be a strictly increasing enumeration of $S$, one still has
		\[
			H = \bigcup_{i \ge 1} H_{k_i}, \ \text{ where } \ H_{k_i} = \{0\} \cup (\Q_{\alpha_{k_i}} \cap b_{k_i}^{-1} \cdot \Z).
		\]
		So after replacing $(\alpha_i)_{i \ge 1}$ by $(\alpha_{k_i})_{i \ge 1}$ and $(b_i)_{i \ge i}$ by $(b_{k_i})_{i \ge 1}$, one can assume that $(\alpha_i)_{i \ge 1}$ strictly increases to infinite. In addition, if $b_j = b_{j+1} =  \dots = b_{j+k}$ for some $j,k \in \N$, then $H_j \cup H_{j+1} \cup \dots \cup H_{j+k} = H_j$. Thus, one can also assume that the sequence $(b_i)_{i \ge 1}$ is strictly increasing without affecting the fact that $(\alpha_i)_{i \ge 1}$ strictly increases to infinite. As a result, $\mathcal{A}(H) \cap H_i$ is a non-empty finite set for every $i \in \N$.
		
		We will prove that $\rho_2(H) = \infty$. Fix $N \in \N$ and set $u_0 = \min \mathcal{A}(H) \cap H_1 = \min \mathcal{A}(H)$. Due to the implication (a) $\Rightarrow$ (f) (which has been already established), one obtains that $\mathcal{A}(H)$ is unbounded. Take $u_n, u_{n+1} \in \mathcal{A}(H)$ such that $\mathsf{d}(u_n) = b_n$, $\mathsf{d}(u_{n+1}) = b_{n+1}$, and $u_n/u_0 > N+1$. Since $u_n - u_0 > 0$ and $\mathsf{d}(u_n - u_0) \mid b_{n+1}$, we find that $u_{n+1} + (u_n - u_0) \in H$, that is, $u_0 \mid_H u_n + u_{n+1}$. Let $m u_0$ be the largest multiple of $u_0$ dividing $u_n + u_{n+1}$ in $H$, and write $u_n + u_{n+1} = m u_0 + x$ for some $x \in H$. One can readily check that $\mathsf{d}(u_n + u_{n+1} - m u_0) \nmid  b_n$. So there exists $u'_{n+1} \in \mathcal{A}(H) \cap H_j$ for some $j \ge n+1$ such that $u'_{n+1} \mid_H x$. So $u_n + u_{n+1} = m u_0 + u'_{n+1} + x'$ for certain $x' \in H$. The maximality of $m$ now implies that $x' = 0$, and then we have $u_n + u_{n+1} = m u_0 + u'_{n+1}$. Since $u'_{n+1}$ is an atom satisfying that $\mathsf{d}(u'_{n+1})$ divides $b_{n+1}$, it follows that $u'_{n+1} \in \mathcal{A}(H) \cap H_{n+1}$. Hence
		\[
			m = \frac{u_n}{u_0} + \frac{u_{n+1} - u'_{n+1}}{u_0} \ge \frac{u_n}{u_0} - 1 > N.
		\]
		Notice that $\{2,m+1\} \subset \mathsf{L}(u_n + u_{n+1})$ because $u_n + u_{n+1} = m u_0 + u'_{n+1}$. This, along with the fact that $m > N$, implies that $\rho_2(H) = \infty$, which contradicts the statement of part~(b).
	\end{proof}
\end{example}
\smallskip

Our next result demonstrates that the cardinality of the class of strongly primary Puiseux monoids that are globally tame or locally but not globally tame (see Theorem \ref{thm:a SPPM has finite elasticity iff it is globally tame}) is at least as large as the cardinality of the class of additive subgroups of the rationals.

\begin{theorem} \label{thm:Puiseux-monoids-and-their-quotient-groups}~	
	\begin{enumerate}
		\item Every monoid whose quotient group is a rank-one torsion-free group and that is not a group is isomorphic to a Puiseux monoid.
		
		\item For every subgroup $Q$ of $(\Q,+)$ there is a Puiseux valuation monoid $V$ such that $\mathsf q(V)=Q$.
		
		\item For every nontrivial Puiseux valuation monoid $V$ there exists a strongly primary monoid $H$ such that $\widehat H = V$ and $(H \DP \widehat H) \ne \emptyset$. Under the Continuum Hypothesis, the set of all non-isomorphic strongly primary globally tame Puiseux monoids has, then, the cardinality of the continuum.
		
		\item For every Puiseux valuation monoid $V$ that is not isomorphic to $(\N_0, +)$, there exists a strongly primary Puiseux monoid $H$ such that $\widehat H = V$, $\Lambda (H) < \infty$, and $(H \DP \widehat H)= \emptyset$. Under the Continuum Hypothesis, the set of all non-isomorphic strongly primary locally but not globally tame Puiseux monoids has, then, the cardinality of the continuum.
	\end{enumerate}
\end{theorem}

\noindent
{\it Remark.} If $V$ is a Puiseux valuation monoid isomorphic to $(\N_0,+)$ and $H \subset V$ a submonoid with $\widehat H = V$, then $\widehat H = \widetilde H$ by Proposition~\ref{prop:closure-of-Puiseux-monoids}.\ref{item:prop closure, equality of closures} and \eqref{primary-closure}, $H$ is finitely generated because $\widetilde H$ is finitely generated, and hence $(H \DP \widehat H) \ne \emptyset$ by \cite[Theorem~2.7.13]{Ge-HK06a}.

\begin{proof}
	1. Let $H$ be an additive monoid whose quotient group $\mathsf q(H)$ is a rank-one torsion-free group. Then, by \cite[Section~24]{Fu70}, the group $\mathsf q(H)$ is isomorphic to a subgroup of $(\Q,+)$. Thus, without restriction we may suppose that $H \subset \mathsf q(H) \subset \Q$. If $H \subset \Q_{\ge 0}$ or $H \subset \Q_{\le 0}$, then we are done. On the other hand, if $H$ contains positive and negative rational numbers, then $H$ is a group by \cite[Theorem~2.9]{Gi84}.
	
	\smallskip
	2. If $Q$ is an additive subgroup of the rational numbers, then it immediately follows that $Q \cap\Q_{\ge 0}$ is a valuation monoid with quotient group $Q$.
	
	\smallskip
	3. Let $V$ be a nontrivial Puiseux valuation monoid. Clearly, $H = \{0\} \cup \{v \in V :  v \ge 1\}$ is a submonoid of $V$. Since $v+V \subset H$ for any $v \in V_{\ge 1}$, it follows that $(H \DP V) \ne \emptyset$, which implies that $\widehat H = \widehat V$. By Proposition~\ref{prop:closure-of-Puiseux-monoids}.\ref{item 2:prop closure, equality of closures} we have $\widehat V = V$. Since $0$ is not a limit point of $H^\bullet$, Theorem~\ref{thm:strongly primary characterization when conductor non-empty} implies that $H$ is strongly primary.
	
	To argue the second statement of part~3, let $\mathfrak{c}$ denote the cardinality of the continuum and let $\mathcal{S}$ be the set of all strongly primary Puiseux monoids with non-empty conductors (up to isomorphism). Since every Puiseux monoid is at most countable, $|\mathcal{S}| \le 2^{\aleph_0} = \mathfrak{c}$. To check that $|\mathcal{S}| \ge \mathfrak{c}$, take $G$ and $G'$ to be non-isomorphic subgroups of $(\Q,+)$. Proposition~\ref{prop:closure-of-Puiseux-monoids} ensures that $V = G \cap \Q_{\ge 0}$ and $V' = G' \cap \Q_{\ge 0}$ are Puiseux valuation monoids. As we have seen before, there exist Puiseux monoids $H$ and~$H'$ in $\mathcal{S}$ satisfying that $\widehat{H} = V$ and $\widehat{H'} = V'$. Since $\mathsf{q}(H) = G \not\cong G' = \mathsf{q}(H')$, the Puiseux monoids $H$ and $H'$ are not isomorphic. Hence $|\mathcal{S}| \ge |\mathcal{S}'|$, where $\mathcal{S}'$ denotes the set of all subgroups of $(\Q,+)$ (up to isomorphism). It follows from~\cite[Section~24]{Fu70} that an abelian group is isomorphic to a subgroup $(\Q,+)$ if and only if it is a rank-one torsion-free group. Then \cite[Corollary~85.2]{Fu73} guarantees that $|\mathcal{S}'| = \mathfrak{c}$. Hence $|\mathcal{S}| = \mathfrak{c}$.

	\smallskip
	4. Let $V$ be a Puiseux valuation monoid that is not isomorphic to $(\N_0,+)$. From Proposition~\ref{prop:closure-of-Puiseux-monoids} one obtains that $V = \mathsf q (V) \cap \mathbb Q_{\ge 0}$. On the other hand, it follows from \cite[Theorem~2]{BZ51} that there exists a sequence $(b_i)_{i \ge 1}$ of positive integers such that $b_i \mid b_{i+1}$ for every $i \in \mathbb N$ and $\mathsf q(V)$ is isomorphic to $\bigcup_{i \ge 1} b_i^{-1} \cdot \mathbb Z$. Hence it will suffice to construct a strongly primary Puiseux monoid $H$ with $\widehat{H} = \bigcup_{i \ge 1} b_i^{-1} \cdot \mathbb N_0$ satisfying that $\Lambda(H) < \infty$ and $(H : \widehat{H}) = \emptyset$.
	
	Since $V$ is not isomorphic to $(\N_0,+)$, the sequence $(b_i)_{i \ge 1}$ tends to infinity. Thus, we may assume without restriction that it is strictly increasing with $b_1 \ge 3$. Accordingly, we define, for each $i \in \mathbb N$,
	\[
		\alpha_{i} = i - \frac{1}{b_i}
		\quad\text{and}\quad
		H_i = \{0\} \cup (\mathbb Q_{\ge \alpha_i} \cap b_i^{-1} \cdot \mathbb Z).
	\]
	Note that $\inf \{ \alpha_i : i \in \N \} = \alpha_1 \ge 2/3 > 0$. Following Example~\ref{exa:sophisticate}, we consider the Puiseux monoid $H = \bigcup_{i \ge 1} H_i$. Proposition~\ref{prop:closure-of-Puiseux-monoids}.\ref{item:prop closure, equality of closures}, along with Claim~\ref{claim:sophisticate(1)}.\ref{item:quotient group 2} of Example~\ref{exa:sophisticate}, ensures that
	\[
		\widehat{H} = \bigcup_{i \ge 1} b_i^{-1} \cdot \mathbb N_0.
	\]
	Furthermore, it follows from Claim~\ref{claim:sophisticate(3)} of Example~\ref{exa:sophisticate} that $H$ is strongly primary. It is clear that $\alpha'_i := i + 1/b_i \in \mathcal{A}(H)$ for every $i \in \N$. Then for each $i \in \N$, the equality $2i = \alpha_i + \alpha'_i$ implies that $2 \in \mathsf L(2i)$. Thus, $\Lambda(H) < \infty$ by \cite[Lemma~3.5.2]{Ge-Ro19b} and, therefore, $(H \DP \widehat H) = \emptyset$ by Theorem~\ref{thm:a SPPM has finite elasticity iff it is globally tame}.
	
	As $\Lambda(H) < \infty$, it follows from Theorem~\ref{thm:a SPPM has finite elasticity iff it is globally tame}.\ref{item: strongly primary globally tame part 2} that $H$ is locally tame but not globally tame. Now by mimicking the argument we gave to verify the second statement of part~3, one can verify the second statement of part~4.
\end{proof}
\smallskip

As pointed out in  \eqref{strongly-primary-domains-are-locally-tame} and~\eqref{finitely-primary-monoids-are-locally-tame}, all strongly primary domains and all  finitely primary monoids are locally tame. So far there is precisely one example in the literature of a strongly primary monoid that is not locally tame (it is constructed as a submonoid of a one-dimensional local Noetherian domain, \cite[Proposition~3.7 and Example~3.8]{Ge-Ha-Le07}). Here we construct a strongly primary Puiseux monoid that is not locally tame. For an additive atomic monoid $H$, for $k \in \N$ and $b \in H$, we set
\[
	\mathsf{Z}_{\text{min}}(k,b) := \bigg\{ \sum_{i=1}^j a_i \in \mathsf{Z}(H) \, : \ j \le k, \ b \mid_H \sum_{i=1}^j a_i, \, \text{ and } \, b \nmid_H \sum_{i \in I} a_i \ \text{ for any } \ I \subsetneq \llb 1,j \rrb \bigg\}.
\]
Then we have
\begin{equation} \label{eq:tameness in terms of minimal factorizations}
	\tau(b) = \sup_k \sup_z \big\{ \min \mathsf{L}\big(\pi(z) - b \big) : z \in \mathsf{Z}_{\text{min}}(k,b)\big\}.
\end{equation}

\begin{proposition} \label{prop:strongly primary PM that is not locally tame}
	There exists a strongly primary Puiseux monoid that is not locally tame.
\end{proposition}

\begin{proof}
	Let $(k_i)_{i \ge 0}$ be a strictly increasing sequence of integers such that $k_0 = 0$ and $k_i > \frac{3}{2}k_{i-1}^2$ for each $i \in \N$. Now consider the Puiseux monoid
	\[
		H := \bigcup_{i \ge 0} H_i, \ \ \text{where} \ \ H_i := \big\langle k_i + \frac{m}{2^i} \ : \ m \in \N_0 \big\rangle.
	\]
	Since $H_i = \{0\} \cup \big( \Q_{\ge k_i} \cap 2^{-i} \cdot \Z \big)$ for every $i \in \N_0$, the monoid $H$ belongs to the class of Puiseux monoids constructed in Example~\ref{exa:sophisticate}. As a result, $H$ is strongly primary. Because the sequence $(k_i)_{i \ge 0}$ is strictly increasing, $\{k_i + 1/2^i : i \in \N_0\} \subset \mathcal{A}(H)$ (in particular, $1 \in \mathcal{A}(H)$). As we know by \eqref{local-tameness-versus-tau}, arguing that $H$ is not locally tame amounts to verifying that $\tau(1) = \infty$. For every $i \in \N$ consider the length-$2$ factorization
	\[
		z_i = 2(k_i + 1/2^i) \in \mathsf{Z}(H).
	\]
	Observe that $2(k_i + 1/2^i) - 1 = (2k_i - k_{i-1} - 1) + (k_{i-1} + 1/2^{i-1}) \in H$ for every $i \in \N$. Therefore $1 \mid_H \pi(z_i)$, where $\pi$ is the factorization homomorphism of $H$. In addition, for every $i \in \N$ the fact that $k_i + 1/2^i \in \mathcal{A}(H)$ implies that $z_i \in \mathsf{Z}_{\text{min}}(k,1)$ for each $k \in \N_{\ge 2}$. Now fix $n \in \N_{\ge 2}$, set $x_n = \pi(z_n) - 1$, and take $z'_n \in \mathsf{Z}(x_n)$. Since $\mathsf{d}(x_n) = 2^{n-1}$, the factorization $z'_n$ must contain an atom whose denominator is at least $2^{n-1}$. This, along with the fact that $x_n < 2k_n < k_N$ when $n < N$ (because $4/3 < 2 \le k_n$ and $\frac{3}{2}k_n^2 < k_N$), implies that the largest atom appearing in $z'_n$ is of the form $k_n + m/2^n$ or $k_{n-1} + m'/2^{n-1}$ (for $m,m' \in \N$ such that $2 \nmid mm'$).
	\vspace{2pt}
	
	Let us verify that the largest atom appearing in $z'_n$ cannot be of the form $k_n + m/2^n$ for any $m \in \N$ such that $2 \nmid m$. Suppose, by way of contradiction, that this is not the case. Then because $2k_n > x_n$, the factorization $z'_n$ must contain exactly one copy of the atom $k_n + m/2^n$, and so one can write
	\begin{equation} \label{eq:strongly primary not locally tame 1}
		2k_n + \frac{1}{2^{n-1}} - 1 = k_n + \frac{m}{2^n} + \sum_{i=1}^N u_i,
	\end{equation}
	for some $N \in \N$ and $u_1, \dots, u_N \in \mathcal{A}(H)$ such that $\mathsf{d}(u_i) \le 2^{n-1}$ for each $i \in \llb 1, N \rrb$. After taking $2$-adic valuation in both sides of~(\ref{eq:strongly primary not locally tame 1}) one finds that $2 \mid m$, which is a contradiction.
	
	Let $u$ be an atom appearing in $z'_n$. By the conclusion of the previous paragraph, $\mathsf{d}(u) \le 2^{n-1}$ and so $u \in \bigcup_{i=0}^{n-1} \mathcal{A}(H_i)$. From this, one can deduce (as in Claim~\ref{claim:sophisticate(2)} of Example~\ref{exa:sophisticate}) that $u < 2k_{n-1} + 1$. Then $2(k_n + 1/2^n) - 1 = x_n \le |z'_n| \, (2k_{n-1} + 1)$, along with $2k_n \ge 3k_{n-1}^2 + 1$, yields
	\[
		|z'_n| \ge \frac{2k_n + \frac{1}{2^{n-1}} - 1}{2k_{n-1} + 1} > \frac{2k_n - 1}{3k_{n-1}} \ge k_{n-1}.
	\]
	As a result, $\min \mathsf{L}(\pi(z_n) - 1) = \min \mathsf{L}(x_n) \ge k_{n-1}$. As $n$ was arbitrarily taken in $\N_{\ge 2}$,
	\begin{align*}
		\tau(1) &\ge \sup_{k \ge 2} \, \sup_z \big\{ \min \mathsf{L}( \pi(z) - 1 ) : z \in \mathsf{Z}_{\text{min}}(k,1) \big\} \\
				&\ge \sup_{k \ge 2} \, \sup \big\{ \min \mathsf{L}( \pi(z_n) - 1) : n \in \N_{\ge 2} \big\} \\
				&\ge \sup \, \{ k_{n-1} : n \in \N_{\ge 2}\} = \infty.
	\end{align*}
	We thus see that $\tau(1) = \infty$, which implies that $H$ is not locally tame.
\end{proof}

\medskip
We conclude this section with a couple of main open problems on strongly primary Puiseux monoid.

\begin{problem}
1. Let $H$ be a BF-monoid. By definition, we have that
\[
\rho (H) < \infty \qquad \text{implies} \qquad \rho_k (H) < \infty \quad \text{for all $k \in \N$} \,.
\]
If $H$ stems from a strongly primary domain (Theorem \ref{global-tameness-domains}), or is finitely primary (Theorem \ref{global-tameness-finitely-primary}), or is a strongly primary Puiseux monoid as given in Example \ref{exa:sophisticate}, then the reverse implication holds. On the other hand, it is easy to see that, for a Krull monoid $H$, the finiteness of the local elasticities $\rho_k(H)$ need not imply the finiteness of the elasticity $\rho(H)$; however, we do not know whether the same can happen when $H$ is a strongly primary monoid or even a strongly primary Puiseux monoid.

\smallskip
2. Let $H$ be a strongly primary monoid. By \eqref{enforcing-local-tameness}, $H$ is locally tame provided that

\[ \quad     \Lambda (H) < \infty \quad \text{or} \quad \rho_k (H) < \infty \ \text{for all $k \in \N$} \,.
\]
By Proposition \ref{prop:strongly primary PM that is not locally tame}, there are strongly primary Puiseux monoids that are not locally tame, whence both conditions may fail. Yet, we do not know whether there are any
	locally tame strongly primary monoids that do not  satisfy one of the two properties in~\eqref{enforcing-local-tameness}.
\end{problem}

\section{Sets of lengths of locally tame strongly primary monoids}
\label{sec:sets of lengths of LTSPM}

\smallskip
\subsection{Structure theorem for sets of lengths and unions of sets of lengths} We proceed to study sets of lengths and unions of sets of lengths of locally tame strongly primary monoids. In order to do so we first introduce sets of distances and catenary degrees.

Let $H$ be a multiplicatively written BF-monoid. If $\rho (H)> 1$, then there is  $a \in H$ with $|\mathsf L (a)| > 1$, whence the $n$-fold sumset $\mathsf L (a) + \dots + \mathsf L (a)$ is contained in $\mathsf L (a^n)$. This implies that $|\mathsf L(a^n)| > n$ for every $n \in \N$. For a finite set $L = \{m_1, \ldots, m_k \} \subset \Z$, with $k \in \N_0$ and $m_1 < \dots < m_k$, we denote by $\Delta (L) = \{m_i - m_{i-1} : i \in \llb 2,k \rrb \} \subset \N$  the set of distances of $L$ and by
\[
	\Delta (H) = \bigcup_{L \in \mathscr L (H)} \Delta (L) \ \subset \N
\]
the {\it set of distances} of $H$. By definition, $\Delta (H)=\emptyset$ if and only if $\rho (H)=1$ and if this holds, then $H$ is said to be {\it half-factorial}.

Take $a \in H$ and $N \in \N$. A finite sequence  $z_0,  \ldots, z_k \in \mathsf Z (a)$  is called a {\it (monotone) $N$-chain of factorizations of $a$} if $\mathsf d (z_{i-1}, z_i) \le N$ for each $i \in \llb 1,k \rrb$ (and $|z_0| \le \dots \le |z_k|$ or $|z_0| \ge \dots \ge |z_k|$). The {\it catenary degree} $\mathsf c (a)$ (resp., the {\it monotone catenary degree $\mathsf c_{\mon} (a)$}) is the smallest $N \in \N_0$ such that each two factorizations $z, z' \in \mathsf Z (a)$ can be concatenated by an $N$-chain of factorizations (resp., by a monotone $N$-chain of factorizations). It is easy to see that $\mathsf c (a) \le \mathsf c_{\mon} (a) \le \max \mathsf L(a)$. Then
\[
	\mathsf c(H) = \sup \{\mathsf c(a) : a \in H \} \in \N_0 \cup \{\infty\} \quad \text{and} \quad
	\mathsf c_{\mon} (H) = \sup \{\mathsf c_{\mon}(a) : a \in H \} \in \N_0 \cup \{\infty\}
\]
are the {\it catenary degree}  and the {\it monotone catenary degree} of $H$, respectively. The monoid $H$ is factorial if and only if $\mathsf c(H)=0$. If $H$ is not half-factorial, then it follows from \cite[Sections~1.4 and~1.6]{Ge-HK06a} that
\begin{equation} \label{catenary}
	\min \Delta (H) = \gcd \Delta (H), \quad 2 + \sup \Delta (H) \le \mathsf c (H) \le \mathsf t (H), \quad \text{and} \quad \mathsf c (H) \le \mathsf c_{\mon} (H) \,.
\end{equation}

For every $k \in \N$, we set $\mathscr U_k (H) = \{k\}$ if $H = H^{\times}$ and we set
\[
	\mathscr U_k(H) = \bigcup_{k \in L \in \mathscr L (H)} L
\]
otherwise. We call $\mathscr U_k(H)$ the {\it union of sets of lengths} containing $k$. Then
\[
	\rho_k (H) = \sup \mathscr U_k (H), \quad \text{and we set} \quad \lambda_k (H) = \min \mathscr U_k (H) \,.
\]
The structure of unions of sets of lengths has been studied for a wide range of monoids and domains (see \cite{F-G-K-T17, Ba-Sm18, Fa-Tr18a,Tr19a} for recent progress).

\smallskip
There are various results showing that given sets occur as sets of lengths in primary BF-monoids. Indeed, let $L \subset \N_{\ge 2}$ be a finite set. Then there are a numerical monoid $H$ and an element $a \in H$ such that $L = \mathsf L(a)$ (\cite[Theorem~3.3]{Ge-Sc18e}). For every sufficiently large $s \in \N$, there are a locally tame primary Mori monoid $H_s \subset (\N_0^s, +)$ with non-empty conductor and an element $a \in H_s$ such that $L = \mathsf L(a)$ (\cite[Theorem~4.2]{Ge-Ha-Le07}). Moreover, there is a Puiseux monoid containing each finite subset of $\N_{\ge 2}$ as a set of lengths (\cite[Theorem~3.6]{Go19a}; see also \cite[Theorem~4.6]{Go19d}). However, such a Puiseux monoid is not strongly primary (a close inspection of the given construction reveals that the constructed monoid does not satisfy the property described in Theorem~\ref{thm:structure-of-strongly-primary-and-its-closure}.\ref{item:strongly primary necessary condition}). Indeed, such a phenomenon cannot occur in locally tame strongly primary monoids. Their systems of sets of lengths are well structured as described in the next theorem, and we observe that  their structure is much simpler than the structure of sets of lengths in non-primary monoids (for an overview see \cite[Chapter 4.7]{Ge-HK06a}).

Parts of Theorem~\ref{thm:structure-of-sets-of-lengths-and-unions} are already known. For example, if $H$ is not only locally but also globally tame, then $\rho_k (H) - \rho_{k-1} (H) \le 1+ \mathsf t (H)$ for every $k \ge 2$, whence the unions $\mathscr U_k (H)$ have the form given in Theorem~\ref{thm:structure-of-sets-of-lengths-and-unions} by \cite[Theorems~3.5 and~4.2]{Ga-Ge09b} (there are locally tame monoids with finite sets of distances whose differences $\rho_k (H) - \rho_{k-1}(H)$ are unbounded and whose unions $\mathscr U_k (H)$ do not have that form). Below we offer a fresh approach to these structural results on sets of lengths and their unions. As summarized in Subsection~\ref{background-strongly-primary}, we recall that all strongly primary monoids with non-empty conductor (Lemma~\ref{2.1}), all finitely primary monoids, and  all strongly primary domains are locally tame. For strongly primary Puiseux monoids we refer to Theorems \ref{thm:a SPPM has finite elasticity iff it is globally tame} and \ref{thm:Puiseux-monoids-and-their-quotient-groups}.

\begin{theorem} \label{thm:structure-of-sets-of-lengths-and-unions}
	Let $H$ be a locally tame strongly primary monoid.
	\begin{enumerate}[label=\textup{\arabic{*}.}]
		\item \label{item:structure theorem part 1} We have
		\[
		  		\mathsf c(H) \le \min \left\{ 1 + \Lambda(H),  \max \{ \mathcal M (u)-1, \mathsf t(H, uH^{\times}) \} :  u \in  \mathcal A(H) \right\} \,.
		\]
		In particular, the set of distances $\Delta (H)$ is finite.
		\item \label{item:structure theorem part 2} Suppose that $\Delta (H) \ne \emptyset$ and set $d = \min \Delta (H)$.
		\begin{enumerate}[label=\textup{(\alph{*})}]
			\item There exists $M \in \N_0$ such that, for every $a\in H$,
			\[
				\big( \min \mathsf L (a) + d \cdot \N_0 \big) \cap \llb \min \mathsf L (a)+M, \max \mathsf L (a)-M \rrb \subset \mathsf L (a) \subset \min \mathsf L (a) + d \cdot \N_0 \,.
			\]
			\item There exists $M' \in \N_0$ such that, for every $k \in \N$,
			\[
			 	\big( \lambda_k (H) + d \cdot \N_0 \big) \cap \llb \lambda_k (H)+M', \rho_k (H)-M' \rrb \subset  \mathscr U_k (H) \subset \lambda_k (H) + d \cdot \N_0 \,.
			\]
		 \end{enumerate}
	\end{enumerate}
\end{theorem}

\begin{proof}
	1. Take $a \in H$. To establish the first upper bound on $\mathsf c(H)$, it is sufficient to assume that $\Lambda (H) < \infty$ and to prove that $\mathsf c(a) \le 1 + \Lambda (H)$. To do so, we verify that every factorization $z \in \mathsf Z(a)$ has a $(\Lambda (H)+1)$-chain of factorizations from $z$ to a factorization of length at most $\Lambda (H)$. We proceed by induction on $|z|$. If $|z| \le \Lambda (H)$, then there is nothing to do. Otherwise, write $z = u_1 \cdot \ldots \cdot u_k$ for some $k \in \N$ such that $k \ge \Lambda (H)+1$ and $u_1, \dots, u_k \in \mathcal A (H_{\red})$. Then $u_1 \cdot \ldots \cdot u_{1 + \Lambda (H)}$ has a factorization $x$ of length $|x| \le \Lambda (H)$. Then $z' = x u_{2+ \Lambda (H)} \cdot \ldots \cdot u_k$ is a factorization of $a$ satisfying that $|z'| < |z|$ and $\mathsf d(z,z') \le 1 + \Lambda (H)$. Thus, the assertion follows by the induction hypothesis.

	To show the second upper bound on $\mathsf c (H)$, we choose $u \in \mathcal A (H)$ and set $N = \max \{ \mathcal M (u)-1, \mathsf t (H, uH^{\times})\}$. We prove that $\mathsf c (a) \le N$ for every  $a \in H$, and we proceed by induction on $\max \mathsf L (a)$. Take $a \in H$ and $z, z' \in \mathsf Z(a)$. If $\max \mathsf L (a) \le N$, then $\mathsf c (a) \le \max \mathsf L(a) \le N$. Suppose that $\max \mathsf L (a) > N$. Then $a \in (H \setminus H^{\times})^{\max \mathsf L (a)} \subset (H \setminus H^{\times})^{\mathcal M (u)} \subset uH$, and hence $a = u b$ for some $b \in H$. By definition of tame degree, there exist factorizations $y = ux$ and $y' = ux'$ in $\mathsf Z(a) \cap u \mathsf Z(H)$ such that $\mathsf d (z, ux) \le \mathsf t (H, uH^{\times})$ and $\mathsf d (z', ux') \le \mathsf t (H, uH^{\times})$. Clearly, $x,x' \in \mathsf Z (b)$. Since $\max \mathsf L(b) < \max \mathsf L(a)$, the induction hypothesis implies that there are factorizations $x=x_0, x_1, \dots, x_s=x'$ such that $\mathsf d (x_{i-1},x_i) \le N$ for every $i \in \llb 1,s \rrb$. Therefore $z, ux, ux_1, \ldots, ux_s, z'$ are factorizations of $a$ whose successive distances are bounded by $N$. Hence $\mathsf c (H) \le N$. Since $2 + \sup \Delta (H) \le \mathsf c(H)$, we obtain that $\Delta (H)$ is finite.
	
	\smallskip
	2. By convention, arithmetic progressions can be empty, finite, or infinite. We set $\mathscr U_0(H)=\{0\}$ and note that $\mathscr U_1 (H)=\{1\}$. If $m \in \N_0$, then $\infty \pm m = \infty$ and $\llb m, \infty \rrb = \N_{\ge m}$.
	
	(a) See \cite[Theorem~4.3.6]{Ge-HK06a}.
	
	(b) Let $M \in \mathbb N_0$ be as in part~(a). Since $d \in \Delta (H)$, there is $a \in H$ such that $\{m,m+d\} \subset \mathsf L(a)$ for some $m \in \N$. Then, for every $k \in \N$, the $k$-fold sumset $\mathsf L (a) + \dots + \mathsf L (a)$ is contained in $\mathsf L (a^k)$, whence $\{km, km+d, \ldots, km+kd\} \subset \mathsf L (a^k)$. Consequently, there exist $a_0 \in H$ and $k_0 \in \mathsf L (a_0)$ such that
	\[
		(k_0 + d \cdot \mathbb Z) \cap \llb k_0-M, k_0+M \rrb \subset \mathsf L (a_0) \subset \mathscr U_{k_0}(H).
	\]
	For every $k \ge k_0$, it follows that
	\begin{equation} \label{equ:4.2}
		(k + d \cdot \mathbb Z) \cap \llb k-M, k+M \rrb \subset (k - k_0) + \mathscr U_{k_0}(H) \subset \mathscr U_{k-k_0} (H) + \mathscr U_{k_0}(H) \subset \mathscr U_k (H) \subset k + d \cdot \mathbb Z,
	\end{equation}
	where the last inclusion is obvious when considering that $L \subset k + d \cdot \mathbb Z$ for all $L \in \mathscr L(H)$ containing $k$.
	
	Now, fix an index $k \ge k_0$ and let $(L_i)_{i \ge 1}$ be a sequence of sets from $\mathscr L(H)$ such that $k \in \bigcap_{i \ge 1} L_i$ and $\lim_{i \to \infty} \max L_i = \rho_k(H)$. For each $i \in \mathbb N$, part~(a) guarantees that
	\[
		(k + d \cdot \mathbb Z) \cap \llb k + M, \max L_i - M \rrb \subset (k + d \cdot \mathbb Z) \cap \llb \min L_i + M, \max L_i - M \rrb \subset L_i \subset \mathscr U_k(H).
	\]
	By the assumptions on the sequence $(L_i)_{i \ge 1}$, this implies at once that
	\begin{equation} \label{equ:4.3}
		(k + d \cdot \mathbb Z) \cap \llb k + M, \rho_k(H) - M \rrb = \bigcup_{i \ge 1} \bigl((k + d \cdot \mathbb Z) \cap \llb k + M, \max L_i - M \rrb\bigr) \subset \mathscr U_k(H).
	\end{equation}
	Likewise, there exists $L \in \mathscr L(H)$ with $\{\lambda_k(H), k\} \subset L$. Then we have
	\begin{equation} \label{equ:4.4}
		(k + d \cdot \mathbb Z) \cap \llb \lambda_k(H) + M, k - M \rrb \subset (k + d \cdot \mathbb Z) \cap \llb \lambda_k(H) + M, \max L - M \rrb \subset L \subset \mathscr U_k(H).
	\end{equation}
	So, putting all the pieces together and noting that $\lambda_k(H) + d \cdot \mathbb Z = k + d \cdot \mathbb Z$, we conclude from~\eqref{equ:4.2},~\eqref{equ:4.3}, and~\eqref{equ:4.4} that
	\[
		(\lambda_k(H) + d \cdot \mathbb N_0) \cap \llb \lambda_k(H) + M, \rho_k(H) - M \rrb \subset \mathscr U_k(H) \subset \lambda_k(H) + d \cdot \mathbb N_0.
	\]
	With this said, we set
	\[
		\rho^\ast = \max\{\rho_j(H): j \in \llb 0, k_0 - 1 \rrb \text{ and } \rho_j(H) < \infty\}
		\quad \text{and} \quad
		M^\prime = \max \{M+k_0, \rho^\ast\}.
	\]
	Then it is clear from the above that, for every $k \in \mathbf N_0$,
	\[
		(\lambda_k(H) + d \cdot \mathbb N_0) \cap \llb \lambda_k (H)+M^\prime, \rho_k (H)-M^\prime \rrb
		\subset \mathscr U_k(H) \subset \lambda_k (H) + d \cdot \N_0 \,,
	\]
	 which finishes the proof.
\end{proof}

As we have seen in Theorem~ \ref{thm:structure-of-sets-of-lengths-and-unions}, being strongly primary and being locally tame are sufficient conditions for a monoid $H$ with non-empty $\Delta(H)$ to have well-structured sets of lengths and unions of sets of lengths. However, there are atomic (Puiseux) monoids with well-structured sets of lengths and unions of sets of lengths that are neither strongly primary nor locally tame. The following example sheds some light upon this observation.

\begin{example}
	For a non-empty finite subset $B$ of $\Q_{> 0}$, consider the Puiseux monoid
	\[
		M_B := \langle b^n :  b \in B \text{ and } n \in \N_0 \rangle.
	\]
	The monoid $M_B$ is called the \emph{rational multicyclic monoid} over $B$ if it is minimally generated by $B$ in the sense that $M_{B'} \subsetneq M_B$ whenever $B' \subsetneq B$. These Puiseux monoids were recently studied in \cite{hP20}, where special emphasis was put on the structure of their sets of lengths (the case where $|B| = 1$ was previously investigated in \cite{Ch-Go-Go20}). A rational multicyclic monoid over $B$ is said to be \emph{canonical} if $B \cap \N = \emptyset$ and $\gcd(\mathsf{d}(b_1), \mathsf{d}(b_2)) = 1$ for all $b_1, b_2 \in B$ with $b_1 \neq b_2$. 
	
	Let $M_B$ be an atomic canonical rational multicyclic monoid. It follows from \cite[Theorem~4.9]{hP20} that the sets of lengths of $M_B$ are unions of finitely many multi-dimensional arithmetic progressions, being arithmetic progressions if $|\mathsf{n}(b_1) - \mathsf{d}(b_1)| = |\mathsf{n}(b_2) - \mathsf{d}(b_2)|$ for all $b_1, b_2 \in M_B$. The special case where $|B| = 1$ was established in \cite[Theorem~3.3]{Ch-Go-Go20}. In addition, \cite[Proposition~4.11]{hP20} guarantees the existence of $N \in \N$ such that for every sufficiently large $k \in \N$, $\mathscr{U}_k(M_B)$ is an almost arithmetic progression with difference $\min \Delta(M_B)$ and bound $N$. If $B = \{r\}$, then it follows from \cite[Corollary~3.4]{Ch-Go-Go20} that $\Delta(M_B) = \{|\mathsf{n}(r) - \mathsf{d}(r)|\}$, while it follows from \cite[Proposition~4.9]{Ch-Go-Go20} that $\mathscr{U}_k(M_B)$ is an arithmetic progression with difference $|\mathsf{n}(r) - \mathsf{d}(r)|$ for every $k \in \N$.

	Let us verify that $M_B$ is not strongly primary. Suppose towards a contradiction that $n M_B \subset 1 + M_B$ for some  $n \in \N$. By \cite[Remark~4.2]{hP20}, $\mathcal{A}(M_B) = \{b^n : b \in B \text{ and } n \in \N_0\}$. Consider the element $x = b + b^2 + \dots + b^n \in nM_B$ for some $b \in B$. It follows from \cite[Remark 4.6]{hP20} that $b + b^2 + \dots + b^n$ is the only factorization of $x$, which contradicts that $x \in 1 + M_B$. Thus, $M_B$ is not strongly primary.
	
	Finally, we verify that $M_B$ is not locally tame provided that $\max B < 1$. Take $b \in B$, and consider the Puiseux monoid $H := \langle b^i : i \in \N_0 \rangle$. Since $\mathcal{A}(H) = \{b^i : i \in \N_0\}$ by \cite[Theorem~6.2]{Go-Go18}, it follows that $\mathcal{A}(H) \subset \mathcal{A}(M_B)$. Because $b < 1$, it follows from \cite[Proposition~5.3]{Ch-Go-Go20} that $\omega(H,1) = \infty$. This, together with the inclusion $\mathcal{A}(H) \subset \mathcal{A}(M_B)$, guarantees that $\omega(M_B,1) = \infty$. As a result, \cite[Theorem~3.6]{Ge-Ha08a} ensures that $M_B$ is not a locally tame monoid. A characterization for local/global tameness in the case where $|B| = 1$ can be found in \cite[Theorem~5.6]{Ch-Go-Go20}.
\end{example}

\smallskip
In the forthcoming  subsections we discuss the parameters occurring in the structural description of sets of lengths given in Theorem~\ref{thm:structure-of-sets-of-lengths-and-unions}.

\smallskip
\subsection{Initial and end parts} Let all notations be as in Theorem~\ref{thm:structure-of-sets-of-lengths-and-unions}. For every positive integer $k$, the set $\mathscr U_k (H) \cap \llb \lambda_k (H), \lambda_k (H)+M' \rrb$ (resp., $\mathscr U_k (H) \cap \llb \rho_k (H) -M', \rho_k (H) \rrb$) is called the {\it initial part} (resp., the {\it end part}) of the set $\mathscr U_k (H)$. Similarly, for each $a \in H$ the set $\mathsf L(a) \cap \llb \min \mathsf L (a), \min \mathsf L (a)+M \rrb$ (resp., $\mathsf L (a) \cap \llb \max \mathsf L (a) -M, \max \mathsf L(a) \rrb$) is called the {\it initial part} (resp., the {\it end part}) of the set $\mathsf L (a)$. In special cases (such as for numerical monoids generated by arithmetic sequences or Puiseux monoids generated by geometric sequences) very explicit descriptions of sets of lengths and their unions are available (in these cases, the initial and end parts are empty; see \cite{A-C-H-P07a, Bl-Ga-Ge11a, Ch-Go-Go20}). Explicit upper bounds for the constant $M$ are given for some one-dimensional local Noetherian domains $R$ with $(R \DP \widehat R) = \{0\}$ (\cite[Section~4]{Ha02a}) and for numerical monoids (\cite{GG-MF-VT15}).
In a variety of settings there are periodicity results for the initial and end parts. For example, if $H$ is a monoid with accepted elasticity (see Subsection~\ref{elast}), then the initial and end parts of the sets $\mathscr U_k (H)$ repeat periodically (\cite[Theorem 1.2]{Tr19a}). But this assumption is far from being necessary (see Example~\ref{exa:strongly-primary-with-irrational-elasticity}). If $R$ is a one-dimensional local Noetherian domain with non-zero conductor $\mathfrak f = (R \DP \widehat R)$ and finite residue class ring $\widehat R/\mathfrak f$, then $R$ is a C-domain (\cite[Theorem~2.11.9]{Ge-HK06a}) and, for every $a \in R^{\bullet}$, the initial and end parts of the sets $\mathsf L(a^n)$ repeat periodically~(\cite{Fo-Ha06a}).

\smallskip
\subsection{The elasticity} \label{elast} As mentioned in Subsection~\ref{background-strongly-primary} (see Theorem~\ref{global-tameness-finitely-primary}), the elasticity of a finitely primary monoid is finite if and only if the monoid has rank~$1$. Let $H$ be a finitely primary monoid of rank~$1$, namely, $H \subset F = F^{\times} \times \mathcal F (\{p\})$, and let $\mathsf v_p \colon H \to \N_0$ denote the homomorphism onto the value semigroup. Suppose that  $\mathsf v_p ( \mathcal A (H)) = \{n_1, \ldots, n_s\}$ with $1 \le n_1 < \dots < n_s$. Then $\rho (H) = n_s/n_1$, and $\rho (H)$ is accepted provided that $F^{\times}/H^{\times}$ is a torsion group. Hence the elasticity of a finitely primary monoid is either rational or infinite (this phenomenon, sometimes called the rational-infinite elasticity property, holds true in larger classes of monoids and domains; see \cite[Section~5]{Go19d}, \cite[Theorem~2.12]{AA92}, and \cite[Theorem~1.1]{Zh19b}). We should notice that there are examples where the factor group $F^{\times}/H^{\times}$ is not a torsion group and the elasticity is not accepted (\cite[Lemma~4.1 and Example~4.2]{Ge-Zh18a}). Lemma~\ref{lem:elasticity-of-Puiseux-monoids} characterizes the strongly primary Puiseux monoids with accepted elasticity. Example~\ref{exa:strongly-primary-with-irrational-elasticity}, on the other hand, provides the first example of an atomic primary monoid with irrational elasticity (the first result in this direction for Dedekind domains was given in~\cite[Theorem~3.2]{AA92}, while a result in contrast for Krull monoids was given in \cite[Theorem~4.2]{Ge-Gr-Sc-Sc10}). In spite of the irrationality of the elasticity, the monoid $H$ in Example~\ref{exa:strongly-primary-with-irrational-elasticity} is strongly primary with non-empty conductor, whence globally tame by Theorem~\ref{thm:a SPPM has finite elasticity iff it is globally tame}. In addition, its unions $\mathscr U_k (H)$ of sets of lengths are finite intervals for all sufficiently large $k \in \N$.

\begin{example} \label{exa:strongly-primary-with-irrational-elasticity}
	Take $\alpha \in \mathbb R_{\ge 1} \setminus \mathbb Q$, and let $H = \mathbb N_0 \cup \mathbb Q_{> \alpha} \subset \mathbb Q_{\ge 0}$. We set
	\[
		A = \{1\} \cup \{q \in \mathbb Q \setminus \mathbb N: \alpha < q \le 1+\alpha\} \quad\text{and}\quad
		\bar{\alpha} = \alpha - \lfloor \alpha \rfloor.
	\]
	We will often use without further comment that $0 < \bar\alpha < 1$ and $1 + \lfloor \alpha \rfloor = \lceil \alpha \rceil < 1 + \alpha$ (since $\alpha \notin \mathbb Z$).
	
	Clearly $H$ is a Puiseux monoid; and by Propositions~\ref{prop:conductor-of-Puiseux-monoids} and Theorem~\ref{thm:strongly primary characterization when conductor non-empty} it is actually a strongly primary Puiseux monoid with non-empty conductor. So we conclude from Theorems~\ref{thm:a SPPM has finite elasticity iff it is globally tame} and part~(b) of Theorem~ \ref{thm:structure-of-sets-of-lengths-and-unions}.\ref{item:structure theorem part 2} that $H$ is globally tame and the unions $\mathscr U_n (H)$ are well structured. We aim to prove that the unions $\mathscr U_n(H)$ are in fact intervals for all sufficiently large $n \in \N$.

	\begin{claim} \label{4.2:claim-5}
		 $\mathcal A(H) = A$, $\rho(H) = 1 + \alpha$,  $\min \Delta(H) = 1$, and $\{2,\lfloor \alpha \rfloor + 1\} \in \mathscr L (H)$. In particular, if $\alpha  \ge 2$, then $\lfloor \alpha \rfloor - 1 \in \Delta (H)$.
	\end{claim}

	\begin{proof}
		If $q \in \mathbb Q_{> 1+\alpha}$, then $q = (q-1) + 1 \in H^\bullet + H^\bullet$. This shows that $\mathcal A(H) \subset A$. The reverse inclusion is trivial since either $x+y \in \mathbb N_{\ge 2}$ or $x+y > 1 + \alpha$ for all $x, y \in H^\bullet$. Consequently, we conclude from Lemma~\ref{lem:elasticity-of-Puiseux-monoids} that $\rho(H) = \sup \mathcal A(H)/\inf \mathcal A(H) = 1 + \alpha$.
		
		For the rest, observe that $\bar\alpha$ is an irrational number between $0$ and $1$. Accordingly, take $n \in \mathbb N_{\ge 3}$ such that $1/n < \min(\bar{\alpha}, 1 - \bar{\alpha})$, and let $\kappa$ be the largest integer for which $\kappa/n < \bar\alpha$. Then
		\begin{equation}\label{equ:squeezing-ineq}
			1 \le \kappa \le n-2
			\quad\text{and}\quad
			\bar\alpha < (\kappa+1)/n < 1,
		\end{equation}
	and it follows that $a = \lfloor \alpha \rfloor + (\kappa+1)/n$ and $b = 1 + \lfloor \alpha \rfloor + \kappa/n$ are in $A$ (that is, are atoms of $H$). Indeed, it is clear from \eqref{equ:squeezing-ineq} that both $a$ and $b$ are in $\mathbb Q \setminus \mathbb N$; in addition, we have
		\[
			\alpha = \lfloor \alpha \rfloor + \bar\alpha < a < \lfloor \alpha \rfloor + \bar \alpha + \frac{1}{n} < 1 + \lfloor \alpha \rfloor + \frac{1}{n} \le b < 1 + \lfloor \alpha \rfloor + \bar\alpha = 1 + \alpha.
		\]
		On the other hand, one checks that $na + (n-1) \cdot 1 = nb$, with the result that $\{n, 2n-1\} \subset \mathsf L(nb)$ (recall that~$1$ is also an atom of $H$). So $\Delta(H)$ is non-empty, and we infer from~\eqref{catenary} that $\min \Delta(H)$ divides $n-1$. But this is only possible if $\min \Delta(H) = 1$, because $n$ can be any integer $\ge 3$.
		
		Finally, take $\varepsilon \in \mathbb R_{> 0}$ such that $\bar\alpha + 2\varepsilon < 1$ and $2 \bar \alpha + 2 \varepsilon \ne 1$. It is  immediate from the above that $\alpha + \varepsilon$ and $\alpha + \bar\alpha + 2\varepsilon$ are both atoms of $H$. This yields $\{2,  \lfloor \alpha \rfloor + 1\} \subset \mathsf L ( 2(\alpha + \varepsilon))$, because
		\[
			2(\alpha + \varepsilon) = (\alpha + \varepsilon) + (\alpha + \varepsilon) =  (\alpha + \bar\alpha + 2\varepsilon) + \lfloor \alpha \rfloor \cdot 1\,.
		\]
	Moreover, if $2(\alpha + \varepsilon) = u_1 + \cdots + u_n  + k$ for some $n \in \mathbf N_{\ge 2}$, $u_1, \ldots, u_n \in A \setminus \{1\}$, and $k \in \mathbf N$, then $n = 2$ and $k = 0$, or else $
	u_1 + \cdots + u_n  + k > 2\alpha + 1 > 2\alpha + 2\varepsilon$. Consequently, any decomposition of $2(\alpha + \varepsilon)$ into a sum of atoms of $H$ is either of the form $u + v$ with $u, v \in A \setminus \{1\}$ or of the form $u + k$ with $u \in A \setminus \{1\}$ and $k \in \mathbf N^+$ (recall that $\alpha > 1$, and note that $2(\alpha + \varepsilon)$ is not an integer since $2 \bar \alpha + 2 \varepsilon \ne 1$); and in the latter case it is easy to check that $k$ is necessarily equal to $\lfloor \alpha \rfloor$.
	
	So, putting it all together, we find that $\mathsf L ( 2(\alpha + \varepsilon)) = \{2,  \lfloor \alpha \rfloor + 1\}$, which finishes the proof.
	\end{proof}

	\begin{claim} \label{4.2:claim-6}
		Take $n \in \N$ with $n \ge \lceil 2\bar{\alpha}^{-1} \rceil$, and set $\kappa_n = \lfloor n \bar \alpha \rfloor$. Then
		\begin{equation} \label{claim-6-1st-equ}
			\rho_n(H) = n \lceil \alpha \rceil + \kappa_n
			\quad\text{and}\quad
			\llb \rho_n(H)-\kappa_n + 1, \rho_n(H) \rrb \subset \mathscr U_n(H).
		\end{equation}
	\end{claim}

	\begin{proof}
		Note that $n \ge 3$ and $1 \le \kappa_n < n$, fix $i \in \llb 0, \kappa_n-1 \rrb$, and set $a_{n,i} =  \lceil \alpha \rceil + (\kappa_n-i)/n$. We get from Claim~\ref{4.2:claim-5} that $a_{n,i}$ is an atom of $H$, because
		\[
			a_{n,i} \in \mathbb Q \setminus \mathbb N
			\quad\text{and}\quad
			\alpha < \lceil \alpha \rceil + \frac{1}{n} \le a_{n,i} <  \lceil \alpha \rceil + \frac{\kappa_n}{n} <  \lceil \alpha \rceil + \bar\alpha = 1 + \alpha\,.
		\]
		Since $na_{n,i} = n \lceil \alpha \rceil + \kappa_n - i \in \mathbb N^+$ and $1$ is also an atom of $H$, it follows that
		\[
			\{n, n \lceil \alpha \rceil + \kappa_n - i\} \subset \mathsf L(na_{n,i}) \subset \mathscr U_n(H).
		\]
		Then $\rho_n(H) \ge n \lceil \alpha \rceil + \kappa_n$, and we aim to show that this last inequality cannot be strict. For, assume the contrary. Because $\alpha$ is irrational and $\kappa_n = \lfloor n \bar \alpha \rfloor$, we see that $\bar\alpha < (\kappa_n+1)/n$. Consequently, we conclude from Claim~\ref{4.2:claim-5} and \cite[Proposition~1.4.2.3]{Ge-HK06a} that
		\[
			1 + \alpha = \lceil \alpha \rceil + \bar \alpha < \lceil \alpha \rceil + \frac{\kappa_n + 1}{n} \le \frac{\rho_n(H)}{n} \le \rho(H) = 1 + \alpha,
		\]
		which is a contradiction.
	\end{proof}

	\begin{claim} \label{4.2:claim-7}
		Let $n$ be an integer with $n \ge \lceil 3\,\bar \alpha^{-1}\alpha \rceil^2$. Then there exists a unique pair $(\ell_n, r_n) \in \mathbb N \times \mathbb N$ such that
		\begin{equation*}\label{equ:claim6-first-condition}
		\ell_n \ge 2,
		\quad
		n = \ell_n \lceil \alpha \rceil + r_n,
		\quad\text{and}\quad
		\frac{r_n}{\ell_n} < \bar\alpha < \frac{r_n + \lceil \alpha \rceil}{\ell_n - 1}.
		\end{equation*}
		In addition,
		\begin{equation}\label{claim-7-1st-equ}
			\lambda_n(H) = \ell_n
			\quad\text{and}\quad
			\llb \lambda_n(H), \lambda_n(H) + r_n - 1 \rrb \subset \mathscr U_n(H) \quad\text{for every large } n \in \mathbb N.
		\end{equation}
	\end{claim}

	\begin{proof}
		By the division algorithm, we may write $n = q \lceil \alpha \rceil + r_0$, where $q \in \mathbb N$ and $r_0 \in \llb 0, \lceil \alpha \rceil - 1 \rrb$. Note that $q \ge \lceil 3\,\bar\alpha^{-1}\alpha \rceil$, or else $q \lceil \alpha \rceil + r_0 \le \lceil 3\,\bar\alpha^{-1} \alpha \rceil^2 - 1 < n$. Also, observe that if $n = \ell \lceil \alpha \rceil + r$ for some $\ell, r \in \mathbb N_0$, then $q \ge \ell$ and $r = r_0 + (q-\ell) \lceil \alpha \rceil \ge r_0$. Accordingly, consider the function
		\[
			\phi \colon \llb 0, q-1 \rrb \to \mathbb Q_{\ge 0} \ \text{ determined by } \ k \mapsto \frac{r_0 + k \lceil \alpha \rceil}{q-k}.
		\]
		Since $q \ge \lceil 3\bar\alpha^{-1}\alpha \rceil \ge 4$, the discrete interval $\llb 1, q-2 \rrb$ is non-empty, and a simple calculation shows that
		\[
			\phi(1) = \frac{r_0 + \lceil \alpha \rceil}{q - 1} < \bar \alpha < \lceil \alpha \rceil \le \frac{r_0 + (q-2) \lceil \alpha \rceil}{2} = \phi(q-2).
		\]
		Because $\phi$ is strictly increasing, there exists a unique $k_\ast \in \llb 1, q-2 \rrb$ such that 
		\[
			\phi(0) = \frac{r_0}{q} < \phi(k_\ast) < \bar\alpha < \phi(k_\ast+1) \le \phi(q-1). 
		\]
		Set
		$\ell_n = q-k_\ast \in \mathbb N_{\ge 2}$ and $r_n = r_0 + k_\ast \lceil \alpha \rceil \in \mathbb N$, and note from the above that $(\ell_n, r_n)$ is the unique pair of positive integers with the property that
		\begin{equation} \label{equ:meeting-the-conditions}
			\ell_n \ge 2, \quad n = \ell_n \lceil \alpha \rceil + r_n,
			\quad\text{and}\quad
			\frac{r_n}{\ell_n} < \bar\alpha < \frac{r_n + \lceil \alpha \rceil}{\ell_n - 1} \,.
		\end{equation}
		Accordingly, define
		\[
			b_{n,i} = \lceil \alpha \rceil + \frac{r_n - i}{\ell_n},
			\quad\text{for } i \in \llb 0, r_n - 1 \rrb \,.
		\]
		Given $i \in \llb 0, r_n - 1 \rrb$, it follows from Claim~\ref{4.2:claim-5} that $b_{n,i}$ is an atom of $H$, because
		\[
			\alpha < \lceil \alpha \rceil + \frac{1}{\ell_n} \le b_{n,i} \le \lceil \alpha \rceil + \frac{r_n}{\ell_n} < \lceil \alpha \rceil + \bar\alpha = 1 + \alpha\,;
		\]
		in particular, $b_{n,i} \notin \mathbb N$. Moreover, it is clear that $n = \ell_n b_{n,i} + i$. So recalling that $1$ is also an atom of $H$, we find that $\{\ell_n + i, n\} \subset \mathsf L(n) \subset \mathscr U_n(H)$. Consequently,
		\[
			\llb \ell_n, \ell_n + r_n - 1 \rrb \subset \mathscr U_n(H)
			\quad\text{and}\quad
			\lambda_n(H) \le \ell_n.
		\]
		It remains to demonstrate that $\lambda_n(H) \ge \ell_n$. Suppose to the contrary that $\lambda_n(H) \le \ell_n - 1$. Then we infer from~\eqref{equ:meeting-the-conditions} that
		\[
			\frac{n}{\ell_n - 1} = \lceil \alpha \rceil + \frac{\lceil \alpha \rceil + r_n}{\ell_n - 1} > \lceil \alpha \rceil + \bar{\alpha} = 1 + \alpha\,.
		\]
		But in view of Claim~\ref{4.2:claim-5} and \cite[Lemma~3.3.1]{Ga-Ge09b}, this yields
		\[
			\frac{1}{1+\alpha} = \frac{1}{\rho(H)} \le \frac{\lambda_n(H)}{n} \le \frac{\ell_n - 1}{n} < \frac{1}{1 + \alpha},
		\]
		which is a contradiction and finishes the proof.
	\end{proof}

	\begin{claim}\label{4.2:claim-8}
		$\mathscr U_n(H)$ is an interval for all but finitely many $n \in \mathbb N$.
	\end{claim}

	\begin{proof}
		Take $n \in \N$ with $n \ge \lceil 3\,\bar \alpha^{-1} \alpha \rceil^2$. Accordingly, let $\kappa_n$ and $r_n$ be defined as in Claims~\ref{4.2:claim-6} and~\ref{4.2:claim-7}. We have already observed that $H$ is globally tame and strongly primary. Hence it follows from Claim~\ref{4.2:claim-5} and  part~(b) of Theorem~\ref{thm:structure-of-sets-of-lengths-and-unions}.\ref{item:structure theorem part 2} that there exists $M \in \mathbb N_0$ such that $\mathscr U_n(H) \cap \llb \lambda_k(H) + M, \rho_k(H) - M \rrb$ is an interval for all but finitely many $n \in \mathbb N$. We are left to show that
		\[
			\llb \lambda_n(H), \lambda_n(H) + M \rrb \cup \llb \rho_n(H) - M, \rho_n(H) \rrb \subset \mathscr U_n(H)
			\quad\text{for every large } n \in \mathbb N.
		\]
		But this is immediate from~\eqref{claim-6-1st-equ} and~\eqref{claim-7-1st-equ} when considering that $\lim_{n \to \infty} \kappa_n = \lim_{n \to \infty} r_n =  \infty$.
	\end{proof}
\end{example}

\smallskip
\subsection{The set of distances} \label{distances} We know only little about  the set of distances of  locally tame strongly primary monoids. By~\eqref{catenary}, any atomic monoid $H$ with $\Delta (H) \ne \emptyset$ satisfies $\min \Delta (H) = \gcd \Delta (H)$. The standing conjecture is that every finite set $\Delta \subset \N$ with $\gcd \Delta = \min \Delta$ is realizable as the set of distances of a numerical monoid, but this has been proved only for two-element sets~(\cite{Co-Ka17a}).  More is known about the possible minima of sets of distances and, as always, the seminormal case is  special. Indeed, if $R$ is a seminormal one-dimensional local Mori domain, then $R^{\bullet}$ is seminormal finitely primary and for every seminormal finitely primary monoid $H$ we have  $\Delta (H) \subset \{1\}$  (\cite[Lemmas~3.4 and~3.6]{Ge-Ka-Re15a}). In that case all sets of lengths and all their unions are intervals. If we drop the seminormality assumption, then every $d\in \N$ occurs as the minimum of a set of distances. More precisely, for every $d \ge 2$ there is a finitely primary monoid $H$ of rank two with $\min \Delta (H)=d$ (\cite[Example~3.1.9]{Ge-HK06a}). If $H$ is a numerical monoid with $\mathcal A (H) = \{n_1, \dots, n_s\}$ where $s \ge 2$ and $1 < n_1 < \dots < n_s$, then $\min \Delta (H) =  \gcd  \{ n_i - n_{i-1} : i \in \llb 2,s \rrb \}$ (\cite[Theorem~2.9]{B-C-K-R06}).

\smallskip
\subsection{The (monotone) catenary degree} By~\eqref{catenary}, we have $2 + \max \Delta (H) \le \mathsf c (H) \le \mathsf c_{\mon} (H)$. If $H = \langle n_1, n_2 \rangle \subset (\N_0,+)$ is a numerical monoid such that $1 < n_1 < n_2$, then $\Delta (H) = \{n_2-n_1\}$ and $\mathsf c(H)=\mathsf c_{\mon} (H) = n_2$ (\cite[Example~3.1.6]{Ge-HK06a}). Thus, the first inequality can be strict. Catenary degrees of numerical monoids have received some attention in the literature. A survey on computational aspects is given in \cite{GS16a}. The finiteness of the monotone catenary degree of numerical monoids was proved in \cite[Theorem~3.9]{Fo06a}. By  \cite[Theorem~4.2]{ON-Pe18a}, every finite subset $C$ of $\N_{\ge 2}$ with $\max C \ge 3$ can be realized as the set of positive catenary degrees of elements of a numerical monoid. Theorem~\ref{thm:structure-of-sets-of-lengths-and-unions} shows that the catenary degree of locally tame strongly primary monoids is finite. However, there are finitely primary monoids  having infinite monotone catenary degree but, on the other hand, there are conditions on finitely primary monoids enforcing the finiteness of the monotone catenary degree~(\cite{Fo06a}). Let $R$ be a one-dimensional local Noetherian domain with maximal ideal $\mathfrak m$ and non-zero conductor $(R \DP \widehat R)$. If the rank $s$ is at most $2$ and  $|\max ( \widehat R)| \le |R/\mathfrak m| < \infty$, then $R$ is a C-domain and $\mathsf c_{\mon} (R^{\bullet}) < \infty$ (\cite[Corollary~5.7 and Proposition~5.12]{Ge-Re19d}). There are domains $R$ as above  with $s=3$ and $\mathsf c_{\mon} (R^{\bullet})=\infty$ (\cite[Examples~6.3 and~6.5]{Ha09c}).
\bigskip

\section*{Acknowledgments}

We are grateful to an anonymous reviewer for her/his attentive reading and comments which have helped to improve the exposition of the paper.

\providecommand{\bysame}{\leavevmode\hbox to3em{\hrulefill}\thinspace}
\providecommand{\MR}{\relax\ifhmode\unskip\space\fi MR }
\providecommand{\MRhref}[2]{%
  \href{http://www.ams.org/mathscinet-getitem?mr=#1}{#2}
}
\providecommand{\href}[2]{#2}


\begin{thebibliography}{10}

\bibitem{A-C-H-P07a}
J.~Amos, S.T. Chapman, N.~Hine, and J.~Paix{\~a}o, \emph{Sets of lengths do not
  characterize numerical monoids}, Integers \textbf{7} (2007), Paper A50, 8p.

\bibitem{AA92}
D.~D. Anderson and D.~F. Anderson, \emph{Elasticity of factorizations in integral domains}, J. Pure Appl. Algebra \textbf{80} (1992), 217--235.

\bibitem{Ba-Ge-Gr-Sm15}
N.R. Baeth, A.~Geroldinger, D.J. Grynkiewicz, and D.~Smertnig, \emph{A
  semigroup-theoretical view of direct-sum decompositions and associated
  combinatorial problems}, J. Algebra Appl. \textbf{14} (2015), 1550016 (60
  pages).

\bibitem{Ba-Sm18}
N.R. Baeth and D.~Smertnig, \emph{Arithmetical invariants of local quaternion
  orders}, Acta Arith. \textbf{186} (2018), 143 -- 177.

\bibitem{Ba00}
V.~Barucci, \emph{Mori domains}, Non-{N}oetherian {C}ommutative {R}ing
  {T}heory, Mathematics and {I}ts {A}pplications, vol. 520, Kluwer Academic
  Publishers, 2000, pp.~57 -- 73.

\bibitem{Ba-An-Fr00a}
V.~Barucci, M.~D'Anna, and R.~Fr{\"o}berg, \emph{Analytically unramified
  one-dimensional semilocal rings and their value semigroups}, J. Pure Appl.
  Algebra \textbf{147} (2000), 215 -- 254.

\bibitem{Ba-Do-Fo97}
V.~Barucci, D.E. Dobbs, and M.~Fontana, \emph{Maximality {P}roperties in
  {N}umerical {S}emigroups and {A}pplications to {O}ne-{D}imensional
  {A}nalytically {I}rreducible {L}ocal {D}omains}, vol. 125, Memoirs of the
  Amer. Math. Soc., 1997.

\bibitem{BZ51}
R.~A. Beaumont and H.~S. Zuckerman, \emph{A characterization of the subgroups
  of the additive rationals}, Pacific J. Math. \textbf{1} (1951), 169 -- 177.

\bibitem{Bl-Ga-Ge11a}
V.~Blanco, P.~A. Garc{\'i}a-S{\'a}nchez, and A.~Geroldinger,
  \emph{Semigroup-theoretical characterizations of arithmetical invariants with
  applications to numerical monoids and {K}rull monoids}, Illinois J. Math.
  \textbf{55} (2011), 1385 -- 1414.

\bibitem{B-C-K-R06}
C.~Bowles, S.T. Chapman, N.~Kaplan, and D.~Reiser, \emph{On delta sets of
  numerical monoids}, J. Algebra Appl. \textbf{5} (2006), 695 -- 718.

\bibitem{Ch-Go-Go20}
S.~T. Chapman, F.~Gotti, and M.~Gotti, \emph{Factorization invariants of
  {P}uiseux monoids generated by geometric sequences}, Comm. Algebra \textbf{48} (2020), 380 -- 396.

\bibitem{Co-Ka17a}
S.~Colton and N.~Kaplan, \emph{The realization problem for delta sets of
  numerical monoids}, J. Commut. Algebra \textbf{9} (2017), 313 -- 339.

\bibitem{Co-Go19a}
J.~Coykendall and F.~Gotti, \emph{On the atomicity of monoid algebras}, J.
  Algebra \textbf{539} (2019), 138 -- 151.

\bibitem{F-G-K-T17}
Y.~Fan, A.~Geroldinger, F.~Kainrath, and S.~Tringali, \emph{Arithmetic of
  commutative semigroups with a focus on semigroups of ideals and modules}, J.
  Algebra Appl. \textbf{11} (2017), 1750234 (42 pages).

\bibitem{Fa-Tr18a}
Y.~Fan and S.~Tringali, \emph{Power monoids: {A} bridge between factorization
  theory and arithmetic combinatorics}, J. Algebra \textbf{512} (2018), 252 --
  294.

\bibitem{Fo06a}
A.~Foroutan, \emph{Monotone chains of factorizations}, Focus on commutative
  rings research (A.~Badawi, ed.), Nova Sci. Publ., New York, 2006, pp.~107 --
  130.

\bibitem{Fo-Ha06a}
A.~Foroutan and W.~Hassler, \emph{Factorization of powers in $\rm{C}$-monoids},
  J. Algebra \textbf{304} (2006), 755 -- 781.

\bibitem{Fu70}
L.~Fuchs, \emph{Infinite {A}belian {G}roups {I}}, Academic Press, 1970.

\bibitem{Fu73}
\bysame, \emph{Infinite {A}belian {G}roups {II}}, Academic Press, 1973.

\bibitem{Ga-Ge09b}
W.~Gao and A.~Geroldinger, \emph{On products of $k$ atoms}, Monatsh. Math.
  \textbf{156} (2009), 141 -- 157.

\bibitem{GG-MF-VT15}
J.I.~Garc{\'i}a-Garc{\'i}a and M.A.~Moreno-Fr{\'i}as and A.~Vigneron-Tenorio,
\emph{Computation of delta sets of numerical monoids}, Monatsh. Math. \textbf{178} (2015), 457--472.


\bibitem{GS16a}
P.A. Garc{\'i}a-S{\'a}nchez, \emph{An overview of the computational aspects of
  nonunique factorization invariants}, in Multiplicative {I}deal {T}heory and
  {F}actorization {T}heory, Springer, 2016, pp.~159 -- 181.

\bibitem{Ge96}
A.~Geroldinger, \emph{On the structure and arithmetic of finitely primary
  monoids}, Czech. Math. J. \textbf{46} (1996), 677 -- 695.

\bibitem{Ge-Gr-Sc-Sc10}
A.~Geroldinger, D.J. Grynkiewicz, G.J. Schaeffer, and W.A. Schmid, \emph{On the
  arithmetic of {K}rull monoids with infinite cyclic class group}, J. Pure
  Appl. Algebra \textbf{214} (2010), 2219 -- 2250.

\bibitem{Ge-HK94}
A.~Geroldinger and F.~Halter-Koch, \emph{Arithmetical theory of monoid homomorphisms}, Semigroup Forum \textbf{48} (1994), 333 -- 362.

\bibitem{Ge-HK06a}
\bysame, \emph{Non-{U}nique {F}actorizations. {A}lgebraic, {C}ombinatorial and
  {A}nalytic {T}heory}, Pure and Applied Mathematics, vol. 278, Chapman \&
  Hall/CRC, 2006.

\bibitem{G-HK-H-K03}
A.~Geroldinger, F.~Halter-Koch, W.~Hassler, and F.~Kainrath, \emph{Finitary
  monoids}, Semigroup Forum \textbf{67} (2003), 1 -- 21.

\bibitem{Ge-HK-Le95}
A.~Geroldinger, F.~Halter-Koch, and G.~Lettl, \emph{The complete integral
  closure of monoids and domains {II}}, Rend. Mat. Appl., VII. Ser. \textbf{15}
  (1995), 281 -- 292.

\bibitem{Ge-Ha08a}
A.~Geroldinger and W.~Hassler, \emph{Local tameness of $v$-Noetherian monoids},
  J. Pure Appl. Algebra \textbf{212} (2008), 1509 -- 1524.

\bibitem{Ge-Ha-Le07}
A.~Geroldinger, W.~Hassler, and G.~Lettl, \emph{On the arithmetic of strongly
  primary monoids}, Semigroup Forum \textbf{75} (2007), 567 -- 587.

\bibitem{Ge-Ka10a}
A.~Geroldinger and F.~Kainrath, \emph{On the arithmetic of tame monoids with
  applications to {K}rull monoids and {M}ori domains}, J. Pure Appl. Algebra
  \textbf{214} (2010), 2199 -- 2218.

\bibitem{Ge-Ka-Re15a}
A.~Geroldinger, F.~Kainrath, and A.~Reinhart, \emph{Arithmetic of seminormal
  weakly {K}rull monoids and domains}, J. Algebra \textbf{444} (2015), 201 --
  245.

\bibitem{Ge-Re19d}
A.~Geroldinger and A.~Reinhart, \emph{The monotone catenary degree of monoids
  of ideals}, Internat. J. Algebra Comput. \textbf{29} (2019), 419 -- 457.

\bibitem{Ge-Ro19b}
A.~Geroldinger and M.~Roitman, \emph{On strongly primary monoids and domains},
  Commun. Algebra \textbf{48} (2020), 4085 -- 4099.

\bibitem{Ge-Sc18e}
A.~Geroldinger and W.A. Schmid, \emph{A realization theorem for sets of lengths
  in numerical monoids}, Forum Math. \textbf{30} (2018), 1111 -- 1118.

\bibitem{Ge-Zh18a}
A.~Geroldinger and Q.~Zhong, \emph{Long sets of lengths with maximal
  elasticity}, Can. J. Math. \textbf{70} (2018), 1284 -- 1318.

\bibitem{Gi84}
R.~Gilmer, \emph{Commutative {S}emigroup {R}ings}, The University of Chicago
  Press, 1984.

\bibitem{Go19c}
F.~Gotti, \emph{Atomic and antimatter semigroup algebras with rational
  exponents}, { https://arxiv.org/abs/1801.06779}.

\bibitem{Go19b}
\bysame, \emph{Increasing positive monoids of ordered fields are {FF}-monoids},
J. Algebra \textbf{518} (2019), 40 -- 56.

\bibitem{Go17a}
\bysame, \emph{On the atomic structure of {P}uiseux monoids}, J. Algebra Appl. \textbf{16} (2017), 1750126 (20 pages).

\bibitem{Go19d}
\bysame, \emph{On the system of sets of lengths and the elasticity of
  submonoids of a finite-rank free commutative monoid}, J. Algebra {A}ppl. \textbf{19} (2020), 2050137 (18 pages).

\bibitem{fG18}
\bysame, \emph{Puiseux monoids and transfer homomorphisms}, J. Algebra \textbf{516} (2018), 95 -- 114.

\bibitem{Go19a}
\bysame, \emph{Systems of sets of lengths of {P}uiseux monoids}, J. Pure Appl.
  Algebra \textbf{223} (2019), 1856 -- 1868.

\bibitem{Go-Go18}
F.~Gotti and M.~Gotti, \emph{Atomicity and boundedness of monotone Puiseux monoids}, Semigroup Forum \textbf{96} (2018), 536 -- 552.

\bibitem{Go-ON19}
F.~Gotti and C.~O'Neill, \emph{The elasticity of {P}uiseux monoids}, J. Commut.
  Algebra \textbf{12} (2020), 319 -- 331.

\bibitem{HK98}
F.~Halter-Koch, \emph{Ideal {S}ystems. {A}n {I}ntroduction to {M}ultiplicative
  {I}deal {T}heory}, Marcel Dekker, 1998.

\bibitem{HK-Ha-Ka04}
F.~Halter-Koch, W.~Hassler, and F.~Kainrath, \emph{Remarks on the
  multiplicative structure of certain one-dimensional integral domains}, Rings,
  {M}odules, {A}lgebras, and {A}belian {G}roups, Lect. Notes Pure Appl. Math.,
  vol. 236, Marcel Dekker, 2004, pp.~321 -- 331.

\bibitem{Ha02a}
W.~Hassler, \emph{Arithmetical properties of one-dimensional, analytically
  ramified local domains}, J. Algebra \textbf{250} (2002), 517 -- 532.

\bibitem{Ha09c}
\bysame, \emph{Properties of factorizations with successive lengths in
  one-dimensional local domains}, J. Commut. Algebra \textbf{1} (2009), 237 --
  268.

\bibitem{Ka16b}
F.~Kainrath, \emph{Arithmetic of {M}ori domains and monoids{\rm \,:} {T}he
  {G}lobal {C}ase}, Multiplicative {I}deal {T}heory and {F}actorization
  {T}heory, Springer Proc. Math. Stat., vol. 170, Springer, 2016, pp.~183 --
  218.

\bibitem{Lu00}
T.G. Lucas, \emph{Examples built with ${D}+{M}$, ${A} + {XB[X]}$ and other
  pullback constructions}, Non-{N}oetherian {C}ommutative {R}ing {T}heory,
  Mathematics and {I}ts {A}pplications, vol. 520, Kluwer {A}cademic
  {P}ublishers, 2000, pp.~341 -- 368.

\bibitem{ON-Pe18a}
C.~O'Neill and R.~Pelayo, \emph{Realisable sets of catenary degrees of
  numerical monoids}, Bull. Australian Math. Soc. \textbf{97} (2018), 240 --
  245.
  
 \bibitem{hP20}
 H. Polo, \emph{On the sets of lengths of {P}uiseux monoids generated by multiple geometric sequences}, Commun. Korean Math. Soc. (to appear). https://arxiv.org/pdf/2001.06158.pdf

\bibitem{Re13a}
A.~Reinhart, \emph{On integral domains that are $\rm{C}$-monoids}, Houston J.
  Math. \textbf{39} (2013), 1095 -- 1116.

\bibitem{Tr19a}
S.~Tringali, \emph{Structural properties of subadditive families with
  applications to factorization theory}, Israel J. Math. \textbf{234} (2019), 1 -- 35.

\bibitem{Zh19b}
Q.~Zhong, \emph{On the arithmetic of {M}ori monoids and domains}, Glasgow Math.
  Journal \textbf{62} (2020), 313 -- 322.

\end{thebibliography}
\end{document}